\documentclass[11pt,reqno]{amsart}
\usepackage[dvipsnames]{xcolor}
\usepackage[linktocpage]{hyperref}
\hypersetup{	
  colorlinks=true,
  linkcolor=NavyBlue,
  anchorcolor=Blue,
  citecolor=Green,
  urlcolor=NavyBlue
}

\usepackage{yhmath}

\usepackage[margin=3cm]{geometry}

\usepackage[dvipsnames]{xcolor}
\usepackage{lmodern}  

\usepackage[T1]{fontenc}
\usepackage{textcomp}
\usepackage{amsmath, amssymb,mathrsfs}
\usepackage{mathtools}

\usepackage{tikz}
\usetikzlibrary{hobby,arrows}
\usetikzlibrary{shapes.geometric}
\usetikzlibrary{arrows.meta}
\usetikzlibrary{calc}
\usepackage{orcidlink}

\makeatletter
\newcommand*{\rom}[1]{\expandafter\@slowromancap\romannumeral #1@}
\makeatother

\newcommand\N{\ensuremath{\mathbb{N}}}
\newcommand\K{\ensuremath{\mathbb{K}}}
\newcommand\R{\ensuremath{\mathbb{R}}}
\newcommand\Z{\ensuremath{\mathbb{Z}}}

\newcommand\C{\ensuremath{\mathbb{C}}}

\newcommand{\supp}{\operatorname{supp}}

\usepackage{bbm}

\newtheorem{theorem}{Theorem}[section]
\newtheorem{proposition}[theorem]{Proposition}
\newtheorem{corollary}[theorem]{Corollary}
\newtheorem{lemma}[theorem]{Lemma}
\newtheorem{conjecture}[theorem]{Conjecture}
\theoremstyle{definition}
\newtheorem{remark}[theorem]{Remark}
\newtheorem*{question}{Question}

\newtheorem{definition}[theorem]{Definition}

\theoremstyle{plain}



\DeclarePairedDelimiterX\ipd[2]{\langle}{\rangle}{#1\delimsize , #2}

\numberwithin{equation}{section}

\mathtoolsset{showonlyrefs}

\title[Sparse supports of lattice eigenfunctions]{Sparse supports of lattice eigenfunctions:\\
quantitative growth and algebraic rigidity}

\author[Y.~Wang]{Yunlei Wang}

\email{yunlei.wang@lsu.edu}

\address{Department of Mathematics, Louisiana State University, Baton Rouge, LA 70803, USA}


\begin{document}

\begin{abstract} 
We study how sparsely a nonzero discrete harmonic function on the standard
lattice $\mathbb{Z}^d$ can be supported. Let $Q_n^{(d)}=\{-n,\cdots,n\}^d$, and let
$m_d(n)$ denote the least possible value of $|\mathrm{supp}(u)\cap Q_n^{(d)}|$ among discrete harmonic functions
$u:\mathbb{Z}^d\to\mathbb{C}$ with $u(0)\neq0$. For all $n\geq1$, we prove
\begin{equation*}
    m_3(n)\asymp n^2,
    \quad
    c_dn^{d^2/(2d-1)}
    \leq m_d(n)\leq
    (2n+1)^{\lfloor d/2\rfloor+1}
    \quad d\ge 4.
\end{equation*}
These estimates extend the two-dimensional support estimate of
Buhovsky, Logunov, Malinnikova, and Sodin
[Duke Math. J. 171 (2022), 1349--1378] to higher dimensions and obtain sharpness in dimension three. For $d\geq4$,
the lower exponent and the upper one differ by less than
$3/4$ in even dimensions and $1/4$ in odd dimensions. The proof combines Hilbert functions of finite support sets with a
position-translation uncertainty principle. The sharp three-dimensional
bound additionally uses Cayley--Bacharach relations and rigidity of
algebraic curves. 

Finally, for every nonzero lattice eigenfunction with
eigenvalue $\lambda$, the Zariski closure of its full support has dimension
at least $\lceil d/2\rceil$, and at least $\lfloor d/2\rfloor+1$ when
$\lambda\neq0$. Both bounds are optimal.

All proofs resulted from human-guided exploration by GPT-5.6 Sol in Ultra mode and checked by the author.
\end{abstract}

\maketitle

{
\setcounter{tocdepth}{1}
\tableofcontents
}

\section{Introduction}
In this article, we study how sparse the support of an eigenfunction on the discrete lattice $\Z^d$ can be. This problem reveals a relationship between the supports of eigenfunctions and the Hilbert functions of their supports.

\subsection{The support-growth problem}

Let $e_1,\cdots,e_d$ be the standard basis of $\Z^d$. We define
\begin{equation*}
    (A_du)(x):=\sum_{i=1}^d\left(u(x+e_i)+u(x-e_i)\right) \quad\text{and}\quad P:=A_d-2dI. 
\end{equation*}
We call $u$ is a \emph{discrete harmonic function} on $\Z^d$ if $Pu=0$ holds for any $x\in \Z^d$. 

Define
\begin{equation*}
    Q_n^{(d)}:=\lbrace -n,-n+1,\cdots,n\rbrace^d,\quad Q_n^{(d)}(x):=x+Q_n^{(d)}.
\end{equation*}

In this article, we concern the following support-growth question:
\begin{question}
    For every integer $d\ge 2$, what is the largest exponent $\beta_d$ for which there exists a positive constant $c_d$ such that for any harmonic function $u:\Z^d\to \C$ satisfying $u(0)\neq 0$ such that
    \begin{equation*}
        |\supp(u)\cap Q_n^{(d)}|\ge c_d n^{\beta_d} \quad \text{ for every integer } n\ge 1.
    \end{equation*}
\end{question}
To the best of our knowledge, the problem is solved only for $d=2$  Buhovsky--Logunov--Malinnikova--Sodin  as a direct result from a more general theorem in \cite{BLMS}, see Subsection~\ref{subsec:history} in detail. The present paper settles $d=3$ and the optimal exponent remains open for $d\ge 4$.

\subsection{Main results}
For simplicity, we define 
\begin{equation*}
    m_d(n):=\inf\left\{ |\supp(u)\cap Q_n^{(d)} |: u:\Z^d\to \C, \,\,A_du=2du, u(0)\neq 0\right\},\quad d\ge 1.
\end{equation*}
Define 
\begin{equation*}
    K_d=\left\lfloor \frac{d}{2}\right\rfloor+1,\quad d\ge 1.  
\end{equation*}
Set
\begin{equation*}
    \delta_d= \begin{cases}
  1-\dfrac{d}{2(2d-1)},&d\text{ even},\\
  \dfrac{d-1}{2(2d-1)},&d\text{ odd},
  \end{cases} \, \text{ for } d\ge 4 \quad \text{and}\quad  \delta_d=0 \text{ for } d=2,3.
\end{equation*}
Our first result gives a uniform quadratic lower bound for the support growth of harmonic function. 
\begin{theorem}\label{thm-quantitative}
    For every integer $d\ge 3$, there exists a constant $c_d>0$ depending only on $d$, such that
    \begin{equation*}
        c_d n^{K_d-\delta_d}\le m_d(n)\le(2n+1)^{K_d},\quad \forall n\in\N.
    \end{equation*}
\end{theorem}
This result is sharp for dimension three, and nearly sharp for sufficiently large dimensions. The gap $\delta_d$ is uniformly bounded, tends respectively to (and smaller than) $3/4$ and $1/4$ along the even and odd dimensions.

\begin{remark}\label{rmk:equivalence}
    Just like complex harmonic function $u:\Z^d\to \C$, one may also define for real harmonic function
    \begin{equation*}
        m_{d,\R}(n):=\inf\left\{ |\supp(u)\cap Q_n^{(d)} |: u:\Z^d\to \R, \,\,A_du=2du, u(0)\neq 0\right\},\quad d\ge 1,
    \end{equation*}
    and write the previous one as $m_{d,\C}$ to make a distinction. Since real functions are also complex functions, we have $m_{d,\C}(n)\le m_{d,\R}(n)$.
    Conversely, for any complex function $u(0)\neq 0$, set 
    $$v(x)=\operatorname{Re}(\alpha u(x)) \quad \text{with }\alpha=\frac{\overline{u(0)}}{|u(0)|}.$$ 
    Since $A_d$ has real coefficients and the harmonic eigenvalue $2d$ is real,
    \begin{equation*}
        A_dv=2dv, \quad v(0)=|u(0)|\neq 0, \quad \text{and} \supp(v)\subset\supp(u).
    \end{equation*}
    Therefore $m_{d,\C}(n)\ge m_{d,\R}(n)$.
    So we have
    \begin{equation*}
        m_{d,\C}(n)= m_{d,\R}(n).
    \end{equation*}
\end{remark}

Our second result changes to viewpoint of Zariski topology and gives sharp Zariski-dimension of support. 
\begin{theorem}\label{thm-zariski}
    Let $d\ge 2$, $\lambda\in \C$. Then for any nonzero solution of $A_du=\lambda u$, we have
    \begin{equation}\label{eq:zariski-general}
        \dim_\C \overline{\supp(u)}^{\,\mathrm{Zar}}\ge \left\lceil \frac{d}{2}\right\rceil.
    \end{equation}
    Moreover, if $\lambda\neq 0$, then
    \begin{equation}\label{eq:zariski-nonzero}
        \dim_\C \overline{\supp(u)}^{\,\mathrm{Zar}}\ge \left\lfloor\frac{d}{2}\right\rfloor+1.
    \end{equation}
    Both bounds are optimal. In particular, the second one is attained in every dimension by a harmonic function.
\end{theorem}

Both theorems suggest the following reasonable conjecture.
\begin{conjecture}
For every integer $d\ge 4$, there exist constants $C_d\ge c_d>0$ depending only on $d$ such that
\begin{equation*}
    c_dn^{K_d}\le  m_d(n)\le C_d n^{K_d},\quad \forall n\in \N.
\end{equation*}
\end{conjecture}

\subsection{Comment on history}\label{subsec:history}
The study of discrete harmonic functions on lattices or more generally weighted graphs has been ongoing for a long time. We refer to~\cite{uniqueness2023malinnikova} for an overview. For brevity, we present only results that are closely related to our work or similar topics. 

In dimension two, the suport-growth problem has been given by Buhovsky--Logunov--Malinnikova--Sodin \cite{BLMS}. Their results revealed phenomena that differ from those in the continuous case. In Euclidean space $\R^d$, it is well-known that if $u$ is a harmonic function and bounded on the whole space, then it must be a constant function. However, this conclusion can be strengthened in two dimensional lattice $\Z^2$. More precisely, it was proved that: if $|u|$ is bounded on $(1-\varepsilon)$ portion of $\Z^2$ with numerical constant $\varepsilon$ being sufficiently small, then $u$ is a constant function.  This result follows from two more general results \cite[Theorem (A)]{BLMS} and \cite[Theorem (B)]{BLMS}.
In particular,
Theorem (A) implies that 
\begin{equation*}
    \left| \left\lbrace|u|>e^{-an}\max_{Q_n^{(2)}} |u |\right\rbrace\cap Q_{2n}^{(2)}\right|\ge c_2n^2
\end{equation*}
for $n$ sufficiently large, see \cite{uniqueness2023malinnikova} for the survey. This implies the following support estimate
\begin{equation*}
    \left|\supp(u)\cap Q_n^{(2)} \right|\ge c_2 n^2,\quad \forall n\ge 1.
\end{equation*}

The estimate above belongs to the broader study of quantitative unique
continuation on lattices. Earlier results of Guadie and Malinnikova and of
Lippner and Mangoubi established three balls inequalities and propagation
of smallness estimates for discrete harmonic functions
\cite{GuadieMalinnikova2014,LippnerMangoubi2015}. Carleman estimates later
gave related results for a broad class of discrete Schr\"odinger
operators \cite{FernandezBertolinEtAl2021}. These results control the
growth of norms rather than the cardinality of the support. More recently,
the large portion Liouville theorem of
Buhovsky, Logunov, Malinnikova, and Sodin was extended to periodic planar
graphs with periodic positive conductances
\cite{BouRabeeCoopermanGanguly2025}.

Another source of the support problem is Anderson localization.
Bourgain and Kenig used quantitative unique continuation to study the
continuous Anderson model with Bernoulli potential
\cite{BourgainKenig2005}. Ding and Smart developed a probabilistic
substitute on $\Z^2$. With high probability, their estimate forces a
local eigenfunction to be quantitatively large on a set of cardinality
at least of order
\begin{equation*}
    L^{3/2}(\log L)^{-1/2}.
\end{equation*}
They used this estimate to prove localization near the edge of the
spectrum \cite{DingSmart2020}. Li subsequently improved the cardinality
to order $L^2$ in a large disorder regime, with a narrower energy window,
and proved localization outside small neighborhoods of finitely many
exceptional energies \cite{LiLargeDisorder2022}. Hurtado later extended
the probabilistic unique continuation and localization argument to
independent potentials that need not be identically distributed
\cite{Hurtado2026}.

In dimension three, Li and Zhang proved a deterministic quantitative
unique continuation theorem for discrete Schr\"odinger equations with
bounded potential. In the zero potential case, their result implies
\begin{equation*}
    \left|
        \left\{
            x\in Q_n^{(3)}:
            |u(x)|\geq e^{-Cn^3}|u(0)|
        \right\}
    \right|
    \geq c\frac{n^2}{\log n}
\end{equation*}
for all sufficiently large $n$ \cite{LiZhang2022}. Consequently,
\begin{equation*}
    \left|\supp(u)\cap Q_n^{(3)}\right|
    \geq c\frac{n^2}{\log n}.
\end{equation*}
This is the closest earlier estimate to our three dimensional result.
The present work removes the logarithmic loss when the potential is zero.

For general dimensions, Krymskii proved that every nonzero solution of a
stationary discrete Schr\"odinger equation on $\Z^d$ has support dimension
at least
\begin{equation*}
    \log_2 d-7
\end{equation*}
\cite{Krymskii}. His argument applies to arbitrary potentials and gives a
dimension dependent estimate in every dimension.  In the particular case $V=0$, Theorem~\ref{thm-quantitative} improves this to $K_d-\delta_d\asymp \frac{d}{2}$.

Very recently, Li considered the corresponding cardinality problem on
finite boxes with Dirichlet boundary conditions and arbitrary real
potential. Let $s_d(N)$ denote the smallest possible support cardinality
of such a solution on a $d$ dimensional box of radius $N$, under the
condition that the solution is nonzero at the origin. Li proved the
dimension reduction inequality
\begin{equation*}
    s_d(N)\geq s_{d-1}(N).
\end{equation*}
Combining this inequality with the estimate of Li and Zhang gives
\begin{equation*}
    s_4(N)\geq c\frac{N^2}{\log N}.
\end{equation*}
Li also constructed solutions whose support has order
$N^{\lceil d/2\rceil}$ \cite{LiSupport2026}. This finite box problem
differs from ours because it permits an arbitrary potential, whereas we
study a fixed harmonic function on the whole lattice.

In his newest related preprint, Li studied a different discrete unique
continuation problem on a lattice simplex. He proved that a solution of
the complete oriented simplex relations which is nonzero at the balanced
point satisfies
\begin{equation*}
    |\supp(g)|\geq c_m R^{\lceil m/2\rceil}.
\end{equation*}
He also gave constructions showing that this exponent is optimal
\cite{LiSimplex2026}. The case $m=3$ is closely related to the triangular
relations used by Li and Zhang. The simplex model has a different
geometry from the nearest neighbor lattice considered here, so its result
is complementary rather than directly applicable.

\subsection{Organization of the article}

In Section~\ref{sec:overview}, we give an overview of the proofs of our main results.  
In Section~\ref{sec:explicit-construct}, we construct sparse eigenfunctions and prove the corresponding sharpness results.  
In Sections~\ref{sec:proof-any-d} and~\ref{sec:proof-sharp-3}, we prove the quantitative lower bounds for $d\geq4$ and $d=3$, respectively.  
Finally, in Section~6, we prove the Zariski-dimension bounds in Theorem~\ref{thm-zariski}.

\section*{Declaration of AI use}
In several rounds guided by human input, OpenAI Codex (Sol model, ultra reasoning effort) developed all proofs in this paper and an early draft. The AI-assisted work involved around 90 hours of cumulative active runtime. During the latter half of the work, one main chat session worked on the proofs and assigned specific tasks to two supporting chat sessions. The three sessions automatically exchanged messages to share partial results and keep their work synchronized. The substantive mathematical problem-solving was performed by the AI model. Yunlei Wang is the sole named human author. He provided mathematical guidance, rewrote the earlier drafts, independently checked every argument and citation, ensured the paper's accuracy and integrity, and accepts full responsibility for its content.

\section{Proof overview}\label{sec:overview}
The upper bounds in Theorem \ref{thm-quantitative} and the sharpness results in our main results come from the explicit construction of eigenfunctions in Section~\ref{sec:explicit-construct}. We only review the lower bound proofs.
\subsection{The quantitative lower bound for
\texorpdfstring{$d\ge 4$}{}}

Set
\begin{equation*}
    R=\left\lfloor\frac n4\right\rfloor,\quad
    E=\supp(u)\cap Q_n^{(d)},\quad
    J=J_R(E),\quad
    A=\K[V_1,\ldots,V_d]/J,\quad
    H(t)=H_A(t).
\end{equation*}
Here $J_R(E)$ is the top-form ideal given in \eqref{eq:top-form-ideal}.
The first key observation is to notice the relation between support set and the Hilbert function. Lemma~\ref{lem:footprint} gives
\begin{equation*}
    H(t)=H_E(t)\leq |E|,\quad 0\leq t\leq R.
\end{equation*}
Hence it suffices to bound the Hilbert function $H$ from below. The crucial feature of $J_r(E)$ is that it controls both position and translation. On the position side, the polynomials vanishing on $E$ allow polynomial coefficients to be reduced to the standard monomials of $J_r(E)$. On the translation side, Lemma~\ref{lem:commutator} shows that if $f|_E=0$, then its top homogeneous part gives
\begin{equation*}
    \operatorname{top}(f)(V)u=0,
    \quad
    V_i=T_i-T_i^{-1}
\end{equation*}
on a smaller region. This follows by iterating commutators with $P$. Both position and translation reductions arise from the same relation
\begin{equation*}
    f|_E=0
    \Longrightarrow
    \begin{cases}
        f(x+h)T^hu=0,
        & \text{position reduction},\\
        \operatorname{top}(f)(V)u=0,
        & \text{translation reduction}.
    \end{cases}
\end{equation*}

Consider the span $\mathcal W_{a,b}(G_L)$ of the restrictions of
\begin{equation*}
    p(x)q(T)u(x),\quad
    \deg p\leq a,\quad
    q(T)=\sum_{\lVert h\rVert_1\leq b}c_hT^h.
\end{equation*}
Here $a,b,R$ satisfy the margin condition~\eqref{eq:margin}.
Let $\mathcal S_a$ be the standard monomials of $J$ of degree at
most $a$. 

Position reduction gives
\begin{equation*}
    \left.p(x)q(T)u\right|_{G_L}
    =
    \left.
    \sum_{m\in\mathcal S_a}m(x)q_m(T)u
    \right|_{G_L},
    \quad
    |\mathcal S_a|=H_A(a).
\end{equation*}
Thus $H_A(a)$ counts the position components.

For the translation reduction, use the free $\K[V]$-basis
\begin{equation*}
    e_\epsilon=\prod_{i=1}^dT_i^{\epsilon_i},
    \quad
    \epsilon\in\{0,1\}^d,
    \quad
    q_m(T)=\sum_{\epsilon\in\{0,1\}^d}
    a_{m,\epsilon}(V)e_\epsilon,
    \quad
    \deg a_{m,\epsilon}\leq b.
\end{equation*}
The choice of $\epsilon$ contributes $2^d$.
The first form of the position--translation principle uses only the top homogeneous part $R_{d,\infty}$ of the determinant norm. Theorem~\ref{thm:uncertainty} gives
\begin{equation*}
    (L+1)^d
    \leq
    \dim\mathcal W_{a,b}(G_L)
    \leq
    2^dH_J(a)H_{J+(R_{d,\infty})}(b).
\end{equation*}
Here $H_J(a)$ counts the position components, while
$H_{J+(R_{d,\infty})}(b)$ counts the translation coefficients after
reduction by the top-form relation. This inequality already exhibits
the position--translation decomposition, but it loses the lower-degree
terms of $\mathcal R_d$. To obtain the stronger estimate needed for the
final exponent, we refine it further to obtain Theorem~\ref{thm:full-provenance} by considering the full,
inhomogeneous norm $\mathcal R_d$, thus gives a better estimate.

Since the adjugate
identity makes $\mathcal R_d$ a global annihilator, the remaining
coefficient space is
\begin{equation*}
    C_b=
    \frac{A_{\leq b}}{\mathcal R_dA_{\leq b-D_d}},
    \quad
    D_d=\deg\mathcal R_d=2^d,
    \quad
    \nu_b=\dim_\K C_b.
\end{equation*}
Choose representatives $r_1,\ldots,r_{\nu_b}$ of a translation basis of $C_b$. Every coefficient has a decomposition
\begin{equation*}
    a_{m,\epsilon}(V)=\sum_{\ell=1}^{\nu_b} \gamma_{m,\epsilon,\ell}r_\ell (V)+\mathcal{R}_d(V)h_{m,\epsilon}(V)+j_{m,\epsilon}(V).
\end{equation*}
Modulo the local $J(V)$ relations and the global $\mathcal R_d$
relation, every element of $\mathcal W_{a,b}(G_L)$ is a linear
combination of
\begin{equation*}
    \left.
    \sum_{m\in\mathcal S_a}
    \sum_{\epsilon\in\{0,1\}^d}
    \sum_{\ell=1}^{\nu_b}
    \gamma_{m,\epsilon,\ell}
    \underbrace{m(x)}_{\text{position}}\,
    \underbrace{r_\ell(V)}_{\text{translation coefficient}}\,
    \underbrace{e_\epsilon}_{\text{Laurent basis}}\,
    u
    \right|_{G_L}.
\end{equation*}
Rank and nullity give
\begin{equation*}
    \nu_b=
    H_A(b)-H_A(b-D_d)
    +
    k_A(b-D_d).
\end{equation*}
Together again with multivariable Lagrange interpolation, this yields
\begin{equation}\label{eq:overview-full-ineq}
\begin{aligned}
    (L+1)^d
    \le \dim\mathcal W_{a,b}(G_L)\le \underbrace{H_A(a)}_{\substack{\text{position}\\\text{components}}}
    \underbrace{2^d}_{\substack{\text{Laurent}\\\text{basis}}}
    \left[
        \underbrace{H_A(b)-H_A(b-D_d)}
        _{\substack{\text{ordinary translation}\\\text{boundary}}}
        +
        \underbrace{k_A(b-D_d)}
        _{\substack{\text{exceptional translation}\\\text{kernel}}}
    \right].
\end{aligned}
\end{equation}

It remains to estimate $k_A(b-D_d)$, which is given by
Proposition~\ref{prop:norm-defect}. We skech the proof of this estimate. Put
\begin{equation*}
    \mathfrak m=(V_1,\ldots,V_d),\quad
    A'=A/\mathfrak m^{t+1}A,\quad
    \mathcal K_t=
    \left(0:_{A'[z]}\widetilde{\mathcal R}_d\right).
\end{equation*}
Then $\dim_\K A'=H_A(t)$, and $\mathcal K_t$ is a torsion-free,
hence free, $\K[z]$-module. After inverting $z$, the change
$V_i\mapsto zV_i$ identifies its rank with the cokernel dimension
of multiplication by $\mathcal R_d$ on $A'_{\K(z)}$:
\begin{equation*}
    k_A(t)
    \leq
    \dim_\K(\mathcal K_t)_t
    \leq
    \operatorname{rank}_{\K[z]}\mathcal K_t
    =
    \dim_{\K(z)}
    \frac{A'_{\K(z)}}{\mathcal R_dA'_{\K(z)}}.
\end{equation*}

Introduce the signed-permutation invariants
$p_i=\sum_jV_j^{2i}$. The ring $\K[V]$ is free of rank $2^dd!$
over $\K[p_1,\ldots,p_d]$, and scalar dilation gives $p_i$ the
distinct positive weight $2i$. In the nilpotent algebra $A'$, the
norm factors as
\begin{equation*}
    \mathcal R_d=U\Gamma,\quad
    \Gamma=
    \sum_{j=1}^d\sqrt{4+V_j^2}-2d
    =
    \sum_{i=1}^da_ip_i
    +O\bigl((p_1,\ldots,p_d)^2\bigr),
    \quad a_i\neq0,
\end{equation*}
where $U$ is a unit. The weighted Green boundary theorem, Theorem~\ref{thm:green-boundary} therefore
gives
\begin{equation*}
    k_A(t)
    \leq
    C_d\bigl(\dim_\K A'\bigr)^{(d-1)/d}
    =
    C_dH_A(t)^{(d-1)/d}.
\end{equation*}

For $t_d\leq t\leq R$, the margin conditions hold after choosing the
dimensional constant $t_d$ sufficiently large. Take
$a=b=t$ and $L=\lfloor t/d\rfloor$ in
\eqref{eq:overview-full-ineq}. Proposition~\ref{prop:norm-defect}
then gives
\begin{equation*}
    t^d
    \leq
    C_d\left(H(t)\bigl(H(t)-H(t-D_d)\bigr)+H(t)^{(2d-1)/d}\right).
\end{equation*}
Since
\begin{equation*}
    H(t)\bigl(H(t)-H(t-D_d)\bigr)
    \leq
    H(t)^2-H(t-D_d)^2,
\end{equation*}
summing through $R$ in the $D_d$ residue classes yields
\begin{equation*}
    R^{d+1}
    \leq
    C_d\left[
        H(R)^2+RH(R)^{(2d-1)/d}
    \right].
\end{equation*}
Consequently,
\begin{equation*}
    H(R)\geq c_dR^{d^2/(2d-1)}.
\end{equation*}
Finally $H(R)=H_E(R)\leq |E|$, which proves the required support
bound. 
\subsection{The sharp quadratic bound in diemsnion three} The preceding argument indeed gives the exponent $9/5$ when $d=3$. Hence a different argument is needed to obtain the sharp exponent $2$. By Remark~\ref{rmk:equivalence} the harmonic function may be taken to be real-valued. Denote by $B_R$ a ball with $\ell^1$ radius $R$. 

For a polynomial $q$ with degree less than $R$, the commutator calculation gives
\begin{equation*}
    P^R(qu)=0.
\end{equation*}
Expanding this identity at a center $a$ produces a relation
\begin{equation*}
    \sum_{x\in \supp(u)\cap (a+B_R)} c_R(a-x)u(x)q(x)=0
\end{equation*}
whose coefficients are nonzero at every point of the local support. Consequently, each centered support set satisfies a Cayley--Bacharach condition: a polynomial of degree at most $R-1$ cannot have nonzero value at exactly one point among a set of points. It follows that the low-degree evaluation Hilbert function is controlled by the number of support points in a boundary shell
\begin{equation*}
    H_E(D-1)\le |\supp(u)\cap (B_r\setminus B_{r-D})|, \quad E=\supp(u)\cap B_r.
\end{equation*}

Suppose, toward a contradiction, that
\begin{equation*}
    M=\left|\supp(u)\cap B_n\right|\ll n^2.
\end{equation*}
First, we establish a simple preliminary result $M\geq n+1$. We choose
$D$ comparable to $M/n$ and average over radii
$n/3\leq r\leq2n/3$. This produces a radius $r$ for which both a
width-$D$ shell $B_r\setminus B_{r-D}$ and the outer width-$2$ shell $B_r\setminus B_{r-2}$ contain few support
points. The preceding rank estimate therefore makes $H_E(D-1)$
small. Then applying Macaulay growth, together with the
Cayley--Bacharach condition, we obtain Lemma~\ref{lem:carrier}. It forces
$\supp(u)\cap B_r$ to lie on a reduced algebraic curve of degree
$O(M/n)$.

The final step rules out such a curve carrier. The eigenvalue equation
has the following local ``port'' consequence: from every supported
point and in every coordinate direction, one can reach another
supported point by one of finitely many short forward displacements.
The Cayley--Bacharach relations force many support points to lie on
each irreducible component of the carrier, whereas B\'ezout's theorem
bounds the number of intersections between two distinct translated
components, see Lemma~\ref{lem:packets}. I t follows that It follows that the finite family of curve components must 
be closed under suitable translations in every coordinate direction.

Following these translations around cycles produces stabilizing
vectors $v_1,\ldots,v_d$. After normalizing them, let $M_0$ be the
matrix whose $i$-th row is $v_i/(v_i)_i$. The geometry of the
forward displacements gives
\begin{equation*}
    \operatorname{tr}M=d,
    \quad
    \|M\|^2\leq2d.
\end{equation*}
The trace inequality for singular values and Cauchy--Schwarz imply
\begin{equation*}
    d
    \leq
    \sqrt{\operatorname{rank}M}\,
    \|M\|,
\end{equation*}
and hence
\begin{equation*}
    \operatorname{rank}M
    \geq
    \left\lceil\frac{d}{2}\right\rceil.
\end{equation*}
Therefore every component is invariant under a linear translation
space of dimension at least $\lceil d/2\rceil$. This is impossible
for a curve when $d\geq3$. The contradiction proves
the desired result.

\subsection{The Zariski-dimension bounds}
The proof of Theorem~\ref{thm-zariski} reuses the port mechanism.
Let
\begin{equation*}
    S=\supp(u),
    \qquad
    X=\overline{S}^{\,\mathrm{Zar}},
    \qquad
    k=\dim X,
\end{equation*}
and consider the top-dimensional irreducible components of $X$.
For each supported point $p$ and each coordinate $i$, either
$p+e_i$ is supported, or evaluating the eigenvalue equation at
$p+e_i$ forces another neighbor of $p+e_i$ to be supported. Thus
there is a displacement $h$ from a fixed finite set such that
\begin{equation*}
    p+h\in S,
    \qquad
    h_i>0,
    \qquad
    \sum_{j\neq i}|h_j|\leq h_i.
\end{equation*}

Since the possible displacements and the irreducible components are
both finite, irreducibility shows that, for every top-dimensional
component $V$ and every $i$, there are another component $V'$
and an admissible displacement $h$ such that
\begin{equation*}
    V'=V+h.
\end{equation*}
Following these translations around the finite family of components
produces stabilizing vectors $v_1,\ldots,v_d$. Consequently, each
component $V$ is invariant under the complex linear space
\begin{equation*}
    W=\operatorname{span}_{\C}\{v_1,\ldots,v_d\}.
\end{equation*}
The same matrix estimate used above gives
\begin{equation*}
    \dim X
    \geq
    \dim W
    \geq
    \left\lceil\frac d2\right\rceil,
\end{equation*}
which proves \eqref{eq:zariski-general}. When $d$ is odd, this is
already equal to $\lfloor d/2\rfloor+1$, so it also proves
\eqref{eq:zariski-nonzero}.

It remains to exclude equality when $d=2q$ and $\lambda\neq0$.
If $\dim X=q$, equality must hold throughout the matrix estimates.
The equality conditions are rigid: after relabeling the coordinates,
they pair the coordinate directions and give
\begin{equation*}
    W
    =
    \operatorname{span}_{\C}
    \left\{
        e_{a_j}+\epsilon_je_{b_j}:
        1\leq j\leq q
    \right\},
    \qquad
    \epsilon_j\in\{\pm1\}.
\end{equation*}
Every top-dimensional component is therefore an affine $q$-plane
parallel to a subspace of this form.

Fix one such direction $W$, and retain from $u$ only the part
supported on components parallel to $W$. Denote the resulting
function by $w$. Its support has Zariski dimension $q$, whereas
the support of
\begin{equation*}
    (A_{2q}-\lambda)w
\end{equation*}
is contained in intersections with components of other directions and
therefore has dimension less than $q$.

Lemma~\ref{lem:paired-layer-rigidity} shows that this situation is
possible only when $\lambda=0$. Its proof decomposes $w$ into its
finitely many parallel layers and records them as a Laurent polynomial
whose coefficients are sequences along $W$, modulo sequences with
lower-dimensional support. The equation for
$(A_{2q}-\lambda)w$ becomes a Laurent-polynomial annihilation
relation. If $\lambda\neq0$, its constant coefficient is the unit
$-\lambda$, and then McCoy's annihilator theorem then rules out a nonzero
solution. This contradiction excludes $\dim X=q$, and hence
\begin{equation*}
    \dim X
    \geq
    q+1
    =
    \left\lfloor\frac d2\right\rfloor+1.
\end{equation*}

\section{Sparse product constructions} \label{sec:explicit-construct}
This section presents explicit constructions of eigenfunctions, which give the upper bounds of Theorem \ref{thm-quantitative} and sharpness conclusion in Theorem \ref{thm-zariski}.

\subsection{Construction of eigenfunctions}
We start from an eigenfunction
\begin{equation*}
    \phi(a,b)=(-1)^a\mathbf{1}_{\left\{a+b=0\right\}}, \quad \forall (a,b)\in \Z^2.
\end{equation*}
This can be found in\cite[Remark 1.3]{BLMS} after alternating the second coordinate. It satisfies $ A_2\phi =0$.

First, we use it to construct harmonic functions in arbitrary dimension. For $d\ge 2$, it can be written as
\begin{equation*}
    d=2k+\ell 
\end{equation*}
for some integer $k$ and $\ell\in \{1,2\}$. Let 
\begin{equation*}
    \gamma=2d-2(\ell-1).
\end{equation*}
Then we define the function
\begin{equation}\label{eq:product-contruction}
    u_d(x):=\left(\prod_{j=1}^k \phi(x_{2j-1},x_{2j})\right)\rho^{x_{2k+1}}
\end{equation}
where $\rho:=(\gamma+\sqrt{\gamma^2-4})/2\ge 1$ is a solution of
\begin{equation*}
    \rho+\rho^{-1}=\gamma.
\end{equation*}
For $\ell=2$, the right-hand side of \eqref{eq:product-contruction} is independent of the last coordinate $x_{2k+2}$. The following proposition gives the exact size of supports.
\begin{proposition}
    The function $u_d$ given in \eqref{eq:product-contruction} is real harmonic, $u_d(0)=1$ and 
    \begin{equation*}
        |\supp(u_d)\cap Q_n^{(d)}|=(2n+1)^{\lfloor d/2\rfloor+1}.
    \end{equation*}
\end{proposition}
\begin{proof}
First we show that $u_d$ is real harmonic. Since each shift in coordinate pair $(x_{2j-1},x_{2j})$ only changes the factor $\phi(x_{2j-1},x_{2j})$, we obtain
\begin{equation*}
    \sum_{i=1}^2 \left(u_d(x+e_{2j-i+1})+u_d(x-e_{2j-i+1})\right)=0
\end{equation*}
where we used the fact $A_2 \phi =0$. For the shifts in the remaining term $\rho^{x_{2k+1}}$, it gives
\begin{equation*}
    \rho^{x_{2k+1}+1}+\rho^{x_{2k+1}-1}=(\rho+\rho^{-1})\rho^{x_{2k+1}}=\gamma \rho^{x_{2k+1}}.
\end{equation*}
If $\ell =1$, summing the above together gives
\begin{equation*}
    A_du_d=\gamma u_d =2d u_d.
\end{equation*}
If $\ell=2$, notice that $\gamma=2d-2$ and the shifts of the remaining coordinate $x_{2k+2}$ give $2u_d$ since $u_d$ is constant with respect to $x_{2k+2}$. Hence in this case
\begin{equation*}
    A_d u_d=(\gamma+2)u_d=2du_d.
\end{equation*}
In both cases, $u_d$ is harmonic. It is easy to check that $u_d(0)=1$.

Now we count the support. For each paired block, $\phi(x_{2j-1},x_{2j})$ is nonzero exactly when $x_{2j-1}+x_{2j}=0$. Thus in $[-n,n]^2$ it there exist exactly the $2n+1$ nonzero points. The exponential coordinate is never zero and therefore also $2n+1$ choices. Finally, we have
\begin{equation*}
    |\supp(u_d)\cap Q_n^{(d)}|=(2n+1)^{k+\ell}=(2n+1)^{\lfloor d/2\rfloor+1}
\end{equation*}
where we used $k+\ell =\lfloor d/2\rfloor +1$, which completes the proof.
\end{proof}

\subsection{Sharpness of Zariski dimenion bound}\label{subsec:sharp-zariski}
    Let
    \begin{equation*}
        L=\lbrace z\in \C^d:z_{2j-1}+z_{2j}=0,1\le j\le k\rbrace
    \end{equation*}
    Then the support set $\supp (u_d)$ is given by
    \begin{equation*}
        \supp(u_d)=L\cap \Z^d.
    \end{equation*}
    Hence $\supp(u_d)$ is contained in $(\lfloor d/2\rfloor+1)$-dimensional complex subspace $L$. Moreover, it is Zariski dense in $L$ by a simple induction argument. Hence
    \begin{equation*}
        \dim_\C \overline{\supp (u_d)}^{\,\mathrm{Zar}}=\dim_\C L = \left\lfloor \frac{d}{2}\right\rfloor +1.
    \end{equation*}
    This gives the sharpness of Theorem \ref{thm-zariski} at eigenvalue $2d$.

For arbitrary eigenvalue $\zeta\in \C$, one may construct the eigenfunction of $A_d v_d=\zeta v_d$ in the same way. For $d=2k+1$, take $k$ copies of $\phi$
and a factor $\rho^{x_{2k+1}}$, where
$\rho+\rho^{-1}=\zeta$. For $d=2k$, take $k-1$ copies of $\phi$, the
factor $\rho^{x_{2k-1}}$, and a constant final coordinate, where
$\rho+\rho^{-1}=\zeta-2$. Each quadratic equation
has a nonzero complex root. The same factor-by-factor calculation shows that $A_du=\zeta u$,
and the support closure has dimension $\lfloor d/2\rfloor+1$.

At eigenvalue zero there is one additional even-dimensional example. Let
$d=2k$, then
\begin{equation*}
 v_{2k}(x)=\prod_{j=1}^k\phi(x_{2j-1},x_{2j})
\end{equation*}
satisfies $A_{2k}v_{2k}=0$ and has support closure equal to a $k$-dimensional linear subspace of $\C^d$.
This attains the bound $\lceil d/2\rceil$ in
Theorem~\ref{thm-zariski}. In odd dimension the two dimension bounds in
that theorem coincide, and the preceding nonzero-eigenvalue construction
already supplies the required sharpness example.

\section{Proof of Theorem \ref{thm-quantitative} for \texorpdfstring{$d\ge 4$}{}}\label{sec:proof-any-d}

In this section, we give the exponent lower bound $\alpha_d=d^2/(2d-1)$ for $d\ge 3$ and hence prove Theorem \ref{thm-quantitative} for $d\ge 4$. We prove it over an arbitrary characteristic-zero field $\K$. For the complex-valued problem stated in the introduction, we take $\K=\C$.

\subsection{The translation algebra and its determinant norms}

Let $T_i$ denote the shift $(T_iu)(x)=u(x+e_i)$ for $i\in \lbrace 1,\cdots,d \rbrace$ and put
\begin{equation*}
    V_i:=T_i-T_{i}^{-1} \text{ and } W_i:=T_i+T_i^{-1}.
\end{equation*}
For simplicity, we denote by $\K[T,T^{-1}]:=\K[T_1^{\pm 1}, \cdots, T_d^{\pm 1}]$ the Laurent polynomial ring. A direct computation shows
\begin{equation*}
    T_i^{-1}= T_i-V_i \quad \text{and}\quad T_i^2-V_iT_i-1=0.
\end{equation*}
This implies that $\K[T,T^{-1}]$ is a free module of rank $2^d$ over $\K[V]:=\K[V_1,\cdots,V_d]$ with basis $\lbrace \Pi_iT_i^{\epsilon_i}:\epsilon\in \lbrace 0,1\rbrace^d\rbrace$. Let
\begin{equation*}
    P_\lambda:=\sum_{i=1}^d W_i-\lambda.
\end{equation*}
In particular, the harmonic operator is $P=P_{2d}$.

Let $m_{P_\lambda}$ be multiplication by $P_\lambda$ on the free module $\K[T,T^{-1}]$ over $\K[V]$. Define 
\begin{equation*}
    \mathcal{R}_{d,\lambda}(V):=\det (m_{P_\lambda})\in \K[V]
\end{equation*}
and we call it the \emph{determinant norm}.
For any $F\in \K[V]$, the \emph{top form of $F$}, is defined to be highest-degree homogeneous part, usually written as $\operatorname{top}(F)$.

\begin{lemma}\label{lem:norm-identities}
For every $\lambda\in \K$, the following statements hold.
\begin{enumerate}
    \item The polynomial $\mathcal{R}_{d,\lambda}$ is nonzero, and
    \begin{equation*}
        \mathcal{R}_{d,\lambda}\K[T,T^{-1}]
        \subseteq
        P_\lambda\K[T,T^{-1}].
    \end{equation*}

    \item After passing to a splitting field containing square roots
    $W_i^2=V_i^2+4$, one has
    \begin{equation*}
        \mathcal{R}_{d,\lambda}(V)
        =
        \prod_{\epsilon\in\{\pm1\}^d}
        \left(
            \sum_{i=1}^d \epsilon_i W_i-\lambda
        \right).
    \end{equation*}

    \item All the polynomials $\mathcal{R}_{d,\lambda}$ have the same highest
    homogeneous part. More precisely,
    \begin{equation*}
        R_{d,\infty}(V):=\operatorname{top}(\mathcal{R}_{d,\lambda})=\prod_{\epsilon\in\{\pm1\}^d}
        \left(
            \sum_{i=1}^d \epsilon_i V_i
        \right).
    \end{equation*}

    \item For the harmonic value $\lambda=2d$, on the formal branch
    satisfying $W_i(0)=2$, the lowest nonzero homogeneous part of
    $\mathcal{R}_d:=\mathcal{R}_{d,2d}$ at $V=0$ is a nonzero scalar multiple of $V_1^2+\cdots+V_d^2$.
\end{enumerate}
\end{lemma}
\begin{proof}
    The Laurent algbebra is a domain and $P_\lambda\neq 0$, hence multiplication by $P_\lambda$ is injective and its determinant is nonzero. The adjugate identity
    \begin{equation*}
        \operatorname{adj}(m_{P_\lambda})m_{P_\lambda}=\mathcal{R}_{d,\lambda} I
    \end{equation*}
    gives the displayed containment. Hence (1) is proved.

    Now adjoin the square roots $W_i^2=V_i^2+4$. We write $\omega_i$ for a chosen square root of $V_i^2+4$ in a splitting field. 
    The quadratic equation for $T_i$ 
    \begin{equation*}
        X^2-V_iX-1=0
    \end{equation*}
    has roots
    \begin{equation*}
        t_{i,\pm}=\frac{V_i\pm\omega_i}{2}.
    \end{equation*}
      One may choose one root for every $i$ and the choices are indexed by $\epsilon=(\epsilon_1,\cdots,\epsilon_d)\in \{\pm 1\}^d$
    with 
    \begin{equation*}
        T_i\longmapsto \frac{V_i+\epsilon_i\omega_i}{2}.
    \end{equation*}
    Since $T_i^{-1}=T_i-V_i$, we obtain $W_i=T_i+T_i^{-1}=2T_i-V_i$. This implies that in the $\epsilon$-component, we have
    \begin{equation*}
        W_i\longmapsto \epsilon_i \omega_i.
    \end{equation*}
    Consequently
    \begin{equation*}
        P_\lambda\longmapsto p_\epsilon=\sum_{i=1}^d\epsilon_i \omega_i-\lambda.
    \end{equation*}
    This proves (2).

    Set $D=2^d$ and along the formal generic ray $V_i=tv_i$, then $\omega_i=tv_i+O(t^{-1})$. By part (2), we obtain
    \begin{equation*}
        \mathcal{R}_{d,\lambda}(t v)=t^D\prod_{\epsilon\in\{\pm1\}^d}\left(\sum_{i=1}^d\epsilon_i v_i\right)+O(t^{D-1}).
    \end{equation*}
    Then (3) follows directly from the above.

For (4), let $\mathfrak m=(V_1,\ldots,V_d)$ and work in
$\widehat A=K[[V_1,\ldots,V_d]]$, choosing
\begin{equation*}
    \omega_i=\sqrt{4+V_i^2}
    =2+\frac14V_i^2+O(V_i^4).
\end{equation*}
If $\epsilon$ has exactly $k$ minus signs, then the corresponding
factor
\begin{equation*}
    p_\epsilon=\sum_{i=1}^d\epsilon_i\omega_i-2d
\end{equation*}
satisfies $p_\epsilon(0)=-4k$. Hence only the all-plus factor vanishes,
while the product $U$ of all remaining factors is a unit with
\begin{equation*}
    U(0)=\prod_{k=1}^d(-4k)^{\binom dk}\neq0.
\end{equation*}
Since
\begin{equation*}
    p_+(V)
    =
    \frac14\sum_{i=1}^dV_i^2+O(\mathfrak m^4)
    \quad\text{and}\quad
    U(V)=U(0)+O(\mathfrak m^2),
\end{equation*}
part~\textup{(2)} gives
\begin{equation*}
    \mathcal{R}_{d,2d}(V)
    =
    \frac{U(0)}4
    \left(V_1^2+\cdots+V_d^2\right)
    +O(\mathfrak m^4).
\end{equation*}
Thus the lowest nonzero homogeneous part is a nonzero
scalar multiple of $V_1^2+\cdots+V_d^2$.    
\end{proof}

\subsection{The top-form ideal and Hilbert function}

Fix a finite set $E\subseteq\mathbb Z^d$, and let
\begin{equation*}
    I(E)
    :=
    \{f\in \K[x_1,\ldots,x_d]:f|_E=0\}
\end{equation*}
be its vanishing ideal. For $r\geq0$, define the \emph{homogeneous
top-form ideal}
\begin{equation}\label{eq:top-form-ideal}
    J_r(E)
    :=
    \bigl(
        \operatorname{top}(f):
        f\in I(E),\ \deg f\leq r
    \bigr).
\end{equation}

\begin{definition}\label{def-new-hilbert}
Denote by $S=\K[x_1,\ldots,x_d]$ and, for $t\geq0$, set
\begin{equation*}
    S_{\leq t}:=\{f\in S:\deg f\leq t\}.
\end{equation*}
If $J\subseteq S$ is a homogeneous ideal, set
$J_{\leq t}:=J\cap S_{\leq t}$. The \emph{Hilbert function of
$S/J$} is defined by
\begin{equation*}
    H_J(t)
    :=
    \dim_\K (S/J)_{\leq t}
    =
    \dim_\K\frac{S_{\leq t}}{J_{\leq t}}.
\end{equation*}
For a finite set $E\subseteq \K^d$, its evaluation Hilbert function is
defined by
\begin{equation*}
    H_E(t)
    :=
    \dim_K
    \left\{
        (q(x))_{x\in E}:
        q\in S,\ \deg q\leq t
    \right\}.
\end{equation*}
Equivalently,
\begin{equation*}
    H_E(t)
    =
    \dim_K\operatorname{im}
    \left(
        S_{\leq t}\longrightarrow K^E
    \right).
\end{equation*}
We also refer to $H_E(t)$ as the Hilbert rank of $E$ up to degree $t$.
\end{definition}

\begin{lemma}\label{lem:footprint}
If $0\leq t\leq r$, then
\begin{equation*}
    H_{J_r(E)}(t)=H_E(t)\leq |E|.
\end{equation*}
\end{lemma}
\begin{proof}
Equip $\K[x_1,\cdots,x_d]$ with the total-degree filtration. For every $t\leq r$, the degree-$t$ component of $J_r(E)$ coincides
with the degree-$t$ component of the associated graded ideal of
$I(E)$. Indeed, if $f\in I(E)$ has degree $t$, then
$\operatorname{top}(f)$ is a generator of $J_r(E)$. Conversely, if
$q\in J_r(E)$ is homogeneous of degree $t$, write
\begin{equation*}
    q=\sum_j h_j\operatorname{top}(f_j),
\end{equation*}
where $f_j\in I(E)$ and the $h_j$ are homogeneous with
$\deg h_j+\deg f_j=t$. Then
\begin{equation*}
    g:=\sum_j h_jf_j\in I(E)_{\leq t}
\end{equation*}
has degree-$t$ homogeneous part $q$. Hence, if $q\neq0$, then
$\operatorname{top}(g)=q$.  The filtered quotient
identity now gives the equality.  The evaluation rank is at most $|E|$.
\end{proof}

The second use of the ideal $J_r(E)$ comes from commutators.  Let $M_f$
denote multiplication by $f(x)$.

\begin{lemma}\label{lem:commutator}
Suppose $Pu=0$ globally, $E=\supp(u)\cap Q_n$, and
$f\in I(E)$ has degree $m$.  Then
\begin{equation*}
        \operatorname{top}(f)(V_1,\ldots,V_d)u=0
        \quad\text{on }Q_{n-m}.
\end{equation*}
\end{lemma}

\begin{proof}
Let $D(A)=[P,A]$. Direct calculation gives
\[
    D(M_{x_i})=[P,M_{x_i}]=V_i
    \quad \text{and}\quad
    D(V_i)=0.
\]
Since $D$ is a derivation, repeated application of the Leibniz rule yields
\[
    D^m(M_f)=m!\operatorname{top}(f)(V_1,\ldots,V_d).
\]
On the other hand, $Pu=0$ and the iterated-commutator formula give
\[
    D^m(M_f)u=P^m(M_fu).
\]
Since $f\in I(E)$ and $E=\operatorname{supp}(u)\cap Q_n$, the function
$M_fu=fu$ vanishes on $Q_n^{(d)}$. Since $P$ has propagation radius one,
$P^m(M_fu)$ vanishes on $Q_{n-m}$. Therefore
\[
    m!\operatorname{top}(f)(V_1,\ldots,V_d)u=0
    \qquad\text{on }Q_{n-m},
\]
and division by $m!$ proves the result.
\end{proof}

\subsection{A finite position-translation uncertainty inequality}

For $q\in \K[T,T^{-1}]$, its \emph{Laurent radius} is
\[
    \operatorname{rad}_{L}(q)
    :=
    \max\bigl\{
        \lVert h\rVert_1:c_h\neq0
    \bigr\},
    \quad
    \lVert h\rVert_1
    :=
    |h_1|+\cdots+|h_d|.
\]
We set $\operatorname{rad}_{L}(0):=0$. For $b\geq0$, define
\[
    \K[T,T^{-1}]_{\leq b}
    :=
    \bigl\{
        q\in \K[T,T^{-1}]:
        \operatorname{rad}_{L}(q)\leq b
    \bigr\}=
    \operatorname{span}_K
    \bigl\{
        T^h:\lVert h\rVert_1\leq b
    \bigr\}.
\]

Fix a degree-compatible monomial order $\prec$ on
$\mathbb{K}[x_1,\ldots,x_d]$, meaning that
$\deg x^\alpha<\deg x^\beta$ implies $x^\alpha\prec x^\beta$.
For a nonzero polynomial $f$, let $\operatorname{LM}_\prec(f)$ denote
its largest monomial with respect to $\prec$. The \emph{initial ideal} of an
ideal $J\subseteq\mathbb{K}[x_1,\ldots,x_d]$ is the monomial ideal
\[
\operatorname{in}_\prec(J)
=
\bigl(\operatorname{LM}_\prec(f):0\neq f\in J\bigr).
\]
The monomials not contained in $\operatorname{in}_\prec(J)$ are called
the \emph{standard monomials} of $J$ with respect to $\prec$.

  For
$G_L=\{0,1,\ldots,L\}^d$, define
\begin{equation*}
 \mathcal{W}_{a,b}(G_L)=\operatorname{span}\left\{
   \left.p(x)q(T)u\right|_{G_L}:
   \deg p\le a,\ q\in \K[T,T^{-1}]_{\le b}\right\}.
\end{equation*}

\begin{theorem}\label{thm:uncertainty}
Let $u:\Z^d\to\K$ satisfy $Pu=0$ and $u(0)\ne0$, let
$E=\supp(u)\cap Q_n^{(d)}$, fix $1\le r<n$, and put
$J=J_r(E)$.  Let $n,L,a,b,r\in \N$ satisfy
\begin{equation}\label{eq:margin}
        a,b\ge dL,
        \quad L+b+r+d\le n.
\end{equation}
Then
\begin{equation}\label{eq:uncertainty}
 (L+1)^d
 \le \dim\mathcal{W}_{a,b}(G_L)
 \le 2^d H_J(a)H_{J+(R_{d,\infty})}(b).
\end{equation}
\end{theorem}

\begin{proof}
We separate the lower and upper estimates.

For $z\in G_L$, let $\ell_z$ be the multivariable Lagrange polynomial which is
one at $z$ and zero at the other points of $G_L$.  It has total degree
$dL$, while $T^{-z}$ has Laurent radius at most $dL$.  Hence
\begin{equation*}
        \left.\ell_z(x)T^{-z}u(x)\right|_{G_L}
        =u(0)\,\delta_z.
\end{equation*}
Since $u(0)\neq 0$, these $(L+1)^d$ vectors are independent.  This
proves the lower bound.

We now prove the upper bound by reducing separately the two factors in
$p(x)q(T)u$. First consider the position polynomial $p(x)$ with $\deg p \le a$. Choose a monomial order compatible with total degree, and denote by $\operatorname{in}(J)$ the initial ideal of $J$ simply. Set
\[
    \mathcal S_a
    :=
    \left\{
        x^\alpha:
        |\alpha|\leq a,\ 
        x^\alpha\notin\operatorname{in}(J)
    \right\}.
\]
The residue classes of the monomials in $\mathcal S_a$ form a basis of
$\K[x]_{\leq a}/J_{\leq a}$, and therefore
\begin{equation*}
    |\mathcal S_a|=H_J(a).
\end{equation*}
For a fixed
translation $\|h\|_1\le b$, let $I_h$ be the ideal generated by
\begin{equation*}
        f(x+h),\quad \forall f\in I(E)\text{ with }\deg f\le r.
\end{equation*}
Every homogeneous $g\in J$ is a homogeneous combination of top forms
$\operatorname{top}(f_i)$.  Replacing each top form by $f_i(x+h)$ produces an element
of $I_h$ with top form $g$.  Lifting a homogeneous Gröbner basis of $J$ (see, for example, \cite[\S 2.2]{EneHerzog2012})
therefore gives
\begin{equation*}
        \operatorname{in}(I_h)\supset\operatorname{in}(J).
\end{equation*}
Consequently the same set $\mathcal S_a$, independent of $h$, spans
every filtered quotient $\K[x]_{\le a}/(I_h)_{\le a}$.  For any $f\in I(E)$, we have
\begin{equation*}
        f(x+h)T^hu(x)=T^h(fu)(x)=0
\end{equation*}
on $G_L$. By the hypothesis and $\|h\|_1\le b$, we have $G_L+h\subset Q_n$. The position
coefficient at each Laurent monomial may therefore be
reduced to this common span.  Grouping equal standard monomials shows
that every generator $p(x)q(T)u$ of $\mathcal{W}_{a,b}(G_L)$ can be written as
\begin{equation}\label{eq:required}
    \sum_{m\in \mathcal{S}_a}m(x)q_m(T)u,\quad q_m\in \K [T,T^{-1}]_{\le b}.
\end{equation}
Hence at most $H_J(a)$ terms $m(x)q_m(T)u|_{G_L}$ are required.

Fix $m\in \mathcal{S}_a$. It remains to bound the span of functions $m(x)q(T)u,q\in \K[T,T^{-1}]_{\le b}$. Using the free $\K[V]$-basis $e_\eta:=\prod_{i}T_i^{\eta_i}, \eta=(\eta_i)\in \lbrace 0,1\rbrace^{d}$, we write
\begin{equation*}
    q(T)=\sum_{\eta\in \{0,1\}^d}a_\eta (V)e_\eta.
\end{equation*}
Using $T_i^{-1}=T_i-V_i$ and $T_i^2=V_iT_i+1$, reduction of a monomial $T^h$ produces coefficients of degree at most $\|h\|_1$. Hence $\deg a_\eta \le b$. Indeed, reduce the Laurent radius by $1$ lower will increase the degree of coefficient polynomial in $q(T)$ at most one.  This will ensure that all subsequent local relations remain inside $Q_n^{(d)}$.

Recall that $R_{d,\infty}=\operatorname{top}(\mathcal{R}_{d,\lambda})(V)$ is the highest-degree homogeneous part of $\mathcal{R}_{d,\lambda}$. Put $K_0=J+(R_{d,\infty})$. Choose a homogeneous Gröbner basis of
$K_0$ and denote it by $\mathcal{G}_0$.  Only basis elements of degree at most $b$ can occur in a
degree-$b$ division of $a_\eta$.  For each homogeneous element $g$ of $\mathcal{G}_0$, choose a
homogeneous decomposition
\begin{equation*}
        g=j_g+c_gR_{d,\infty},\qquad j_g\in J,
\end{equation*}
and replace it by the inhomogeneous lift
\begin{equation*}
        \widetilde g=j_g+c_g\mathcal R_d,
        \qquad \operatorname{top}(\widetilde g)=g.
\end{equation*}
Then the division of $a_\eta$ by the lifts $\widetilde g$ gives
\begin{equation*}
    a_\eta(V)=\sum_{g\in \mathcal{G}_0}q_{\eta,g}\widetilde{g}(V) +r_\eta(V).
\end{equation*}
Thus division by the lifts uses the same leading monomials as division
by the Gröbner basis $\mathcal G_0$. Consequently, $r_\eta$ is a linear
combination of standard monomials of $K_0$ of degree at most $b$, and
the division does not increase total degree. In particular, whenever
$q_{\eta,g}\neq0$,
\begin{equation*}
    \deg q_{\eta,g}+\deg g\le b.
\end{equation*}
Moreover,
\begin{equation*}
    q_{\eta,g}\widetilde g
    =
    q_{\eta,g}j_g
    +
    q_{\eta,g}c_g\mathcal R_d.
\end{equation*}
The first term is a $J(V)$-multiple of degree at most $b$, while the
second is an $\mathcal R_d$-multiple. The first annihilates
$u$ on $G_L$. Indeed, every $F(V)\in J(V)$ arising in the division, with
$\deg F\le b$, can be written as
\begin{equation*}
    F(V)=\sum_\nu h_\nu(V)\operatorname{top}(f_\nu)(V),
\quad
\deg h_\nu+\deg f_\nu\le b.
\end{equation*}
Lemma~\ref{lem:commutator} and
$\operatorname{rad}_L(h_\nu(V))\le\deg h_\nu$ imply that
$F(V)u=0$ on $Q_{n-b}$. Since
$\operatorname{rad}_L(e_\eta)\le d$, we have
$F(V)e_\eta u=0$ on $G_L$ by \eqref{eq:margin} and $F(V)e_\eta=e_\eta F(V)$.
Here the fixed free basis $\lbrace e_\eta\rbrace_{\eta\in \{0, 1\}^d}$ adds
at most $d$ to the Laurent radius.  The
second annihilates
globally since $\mathcal {R}_d\K[T,T^{-1}]\subset P\K[T,T^{-1}]$.
Thus for each fixed $m\in\mathcal \mathcal{S}_a$, the translation span has
dimension at most
\begin{equation*}
        2^dH_{J+(R_{d,\infty})}b.
\end{equation*}
Recalling that at most $H_J(a)$ terms are required in \eqref{eq:required} and summing the above bound over all $m\in\mathcal S_a$, we obtain
\[
    \dim\mathcal W_{a,b}(G_L)
    \le
    2^dH_J(a)H_{J+(R_{d,\infty})}(b),
\]
which proves the upper bound in \eqref{eq:uncertainty}.
\end{proof}

\begin{remark}
The two factors in \eqref{eq:uncertainty} involve the same ideal since both reductions
come from $f(x)u(x)=0$ on the observed support.
Arbitrary constant coefficient annihilators would supply the translation factor but not the shifted position quotients, and \eqref{eq:uncertainty} would then be
false.
\end{remark}

\subsection{The full-norm refinement}

Let
\[
        D_d=\deg\mathcal R_d=2^d,
        \qquad \mathcal R_d=\sum_{j=0}^{D_d}\rho_j,
\]
where $\rho_j$ is homogeneous of degree $j$, and homogenize by
\begin{equation}\label{eq:inset-z}
 \widetilde{\mathcal R}_d(V,z)
   =\sum_{j=0}^{D_d}\rho_j(V)z^{D_d-j}.
\end{equation}
For a homogeneous quotient $A=\K[V_1,\ldots,V_d]/J$, write
\[
 H_A(t)=\dim_\K A_{\le t},\qquad
 k_A(t)=\dim_\K(0:_{A[z]}\widetilde{\mathcal{R} }_d)_t,
\]
where
\begin{equation*}
    0:_{A[z]} \widetilde{\mathcal{R}}_d =\lbrace F\in A[z]: \widetilde{\mathcal{R}}_d F=0 \text{ in } A[z] \rbrace.
\end{equation*}
We set both quantities equal to zero at negative arguments.
Homogenization identifies $k_A(t)$ with the kernel of multiplication by
$\mathcal R_d$ on the filtered piece $A_{\le t}$.

\begin{theorem}
\label{thm:full-provenance}
Under the hypotheses and notation of Theorem~\ref{thm:uncertainty}, identify $\K[x]$ with $\K[V]$ by $x_i\mapsto V_i$ and put $A=\K[V]/J$. Then
\begin{equation}\label{eq:full-provenance}
 (L+1)^d\le 2^dH_A(a)
 \left[H_A(b)-H_A(b-D_d)+k_A(b-D_d)\right].
\end{equation}
\end{theorem}

\begin{proof}
The position reduction and the multivariable Lagrange lower bound are exactly
those in Theorem~\ref{thm:uncertainty}.  For the translation reduction,
expand every Laurent word in the fixed free basis
$\prod_iT_i^{\epsilon_i}$ over $\K[V]$.  Modulo the local $J(V)$
relations, a coefficient of filtered degree at most $b$ lies in
$A_{\le b}$.  The adjugate relation
$\mathcal{R}_d \K[T,T^{-1}] \subset P \K[T,T^{-1}]$ shows that
$\mathcal{R}_dA_{\le b-D_d}$ kills $u$ on the observation box.  Hence the
coefficient space has dimension at most
\[
 \dim\left(\frac{A_{\le b}}{\mathcal R_dA_{\le b-D_d}}\right)
 =H_A(b)-H_A(b-D_d)+k_A(b-D_d).
\]
The free basis contributes $2^d$, the common position complement
contributes $H_A(a)$, and the displayed margin keeps every local relation
inside $Q_n^{(d)}$.  This proves \eqref{eq:full-provenance}.
\end{proof}

The point of retaining the full, inhomogeneous norm is that its exceptional
kernel can be bounded more sharply than an arbitrary top-form kernel.  The
needed estimate is the following equivariant boundary theorem. 

Before we start it, we first give some notions and definitions. Set
\begin{equation*}
    \mathscr{P}=\K[p_1,\ldots,p_d],
    \qquad
    \mathfrak{n}=(p_1,\ldots,p_d).
\end{equation*}
If $\mathfrak{n}^qM=0$ for some $q\geq1$, the
\emph{$\mathfrak{n}$-adic filtration} of $M$ is
\begin{equation*}
    M\supseteq\mathfrak{n}M
    \supseteq\mathfrak{n}^2M
    \supseteq\cdots
    \supseteq\mathfrak{n}^qM=0,
\end{equation*}
and its associated graded module is
\begin{equation*}
    \operatorname{gr}_{\mathfrak{n}}M
    :=
    \bigoplus_{j\geq0}
    \frac{\mathfrak{n}^jM}{\mathfrak{n}^{j+1}M}.
\end{equation*}
If
\begin{equation*}
    v\in\mathfrak{n}^jM
    \setminus\mathfrak{n}^{j+1}M,
\end{equation*}
its \emph{initial class} is
\begin{equation*}
    \operatorname{in}(v)
    :=
    v+\mathfrak{n}^{j+1}M
    \in
    \frac{\mathfrak{n}^jM}{\mathfrak{n}^{j+1}M}.
\end{equation*}
For a submodule $L\subseteq M$, its \emph{initial module}
$\operatorname{in}_{\mathfrak{n}}(L)$ is the graded submodule of
$\operatorname{gr}_{\mathfrak{n}}M$ generated by the initial classes
of the nonzero elements of $L$.

The induced filtration on $M/L$ gives
\begin{equation*}
    \operatorname{gr}_{\mathfrak{n}}(M/L)
    \cong
    \frac{
        \operatorname{gr}_{\mathfrak{n}}M
    }{
        \operatorname{in}_{\mathfrak{n}}(L)
    }.
\end{equation*}
Consequently, when $M$ is finite-dimensional over $\K$,
\begin{equation*}
    \dim_{\K}(M/L)
    =
    \dim_{\K}
    \frac{
        \operatorname{gr}_{\mathfrak{n}}M
    }{
        \operatorname{in}_{\mathfrak{n}}(L)
    }.
\end{equation*}

\begin{theorem}\label{thm:green-boundary}
Let $\mathscr{P}$ and $\mathfrak{n}$ be given above,
and let the one-dimensional torus act by
$p_i\mapsto s^{w_i}p_i$, where the positive weights $w_i$ are pairwise
distinct.  Suppose that $M$ is a finite-length equivariant
$\mathscr P$-module generated by at most $r$ elements.  Let $\Gamma\in \mathscr{P}$ be a polynomial with linear part $\ell=\sum_{i=1}^dc_ip_i$ where $c_i\neq 0$, that is,
\[
 \Gamma=\ell+h,\quad \text{for some }h\in \mathfrak{n}^2.
\]
then
\begin{equation}\label{eq:weighted-green-boundary}
 \dim_\K M/\Gamma M
 \le C_{d,r}\bigl(\dim_\K M\bigr)^{(d-1)/d}.
\end{equation}
The statement is unchanged if $\Gamma$ is a formal series, since a power
of $\mathfrak{n}$ annihilates $M$.
\end{theorem}

\begin{proof}
Faithfully flat extension to $\overline\K$ preserves length and the number
of generators, so we may first assume that $\K$ is algebraically closed.
The support of $M$, denoted by
\begin{equation*}
    \supp_{\mathscr P} (M):=\lbrace \mathfrak{p}\in \operatorname{Spec}(\mathscr P): M_{\mathfrak{p}}\neq 0 \rbrace,
\end{equation*}
is then a finite torus-stable set, that is, it consists of finitely many points, and if
$a=(a_1,\ldots,a_d)\in\operatorname{Supp}(M)$, then
\[
    s\cdot a
    :=
    (s^{w_1}a_1,\ldots,s^{w_d}a_d)
    \in\operatorname{Supp}(M)
\]
for every $s\in\K^\times$.
Since the torus is
connected, each orbit in this finite set is a point. However, positive weights have
only the origin as a fixed point.  Thus $M$ is supported at $\mathfrak{n}$
and some power of $\mathfrak{n}$ annihilates it.

Give $M$ its ordinary $\mathfrak{n}$-adic filtration and put $N=\operatorname{gr}_{\mathfrak{n}}M$. The initial module of $\Gamma M$
contains $\ell N$. Indeed, if $\ell\operatorname{in}(m)\ne0$, it is exactly the
initial class of $\Gamma m$, while if it is zero there is nothing to
prove.  Consequently
\begin{equation}\label{eq:green-filtered-bound}
 \dim M/\Gamma M\le\dim N/\ell N.
\end{equation}
The degree-zero space $N_0$ has dimension at most $r$.  Choose a
torus-eigenbasis and adjoin its vectors one at a time.  This filters $N$
by at most $r$ cyclic equivariant quotients
$B_j=\mathscr P/I_j$, with $I_j$ homogeneous and torus-invariant.
For each short exact sequence
$0\to N_{j-1}\to N_j\to B_j\to0$, tensoring with
$\mathscr P/(\ell)$ gives a right-exact sequence; the preceding
$\operatorname{Tor}_1$ term can only shrink the first image.  Hence
\begin{equation}\label{eq:green-cyclic-decomposition}
 \dim N/\ell N\le\sum_j\dim B_j/\ell B_j,
 \qquad \sum_j\dim B_j=\dim M.
\end{equation}

It remains to establish the required bound for each cyclic factor.
Fix $j$, suppress the index, and write
\begin{equation*}
    B:=B_j=\mathscr P/I_j,
    \quad
    e:=\dim_{\K}B_1.
\end{equation*}
Since $(I_j)_1$ is torus invariant and the weights $w_i$ are pairwise
distinct, $(I_j)_1$ is spanned by a subset of the variables $p_i$.
After deleting those variables and relabeling, the nonzero classes
\begin{equation*}
    \overline p_1,\ldots,\overline p_e
\end{equation*}
form a basis of $B_1$.

If $e=0$, then $B=\K$. If $e=1$, then $B$ is a finite quotient of a
polynomial ring in one variable and $\ell$ is a nonzero multiple of
that variable. In either case,
\begin{equation*}
    \dim_{\K}B/\ell B=1,
\end{equation*}
so the required estimate is immediate. We may therefore assume
$e\geq2$.

For $s\in\K^\times$, set
\begin{equation*}
    \ell_s=\sum_{i=1}^e c_i s^{w_i}\overline p_i.
\end{equation*}
The Zariski closure of
$\{\ell_s:s\in\K^\times\}$ is an irreducible curve whose linear span is
$B_1$. Indeed, if a linear form vanishes on every $\ell_s$, then
\begin{equation*}
    \sum_{i=1}^e\alpha_i c_i s^{w_i}=0
    \quad\text{for every }s\in\K^\times.
\end{equation*}
The distinct characters $s^{w_i}$ are linearly independent, so every
$\alpha_i$ is zero. For a fixed degree $n$, Caviglia's
refinement of Green's hyperplane restriction theorem (see \cite{Green1989}) says that the bad
forms in that degree lie in a finite union of proper linear subspaces
\cite[Theorems 1.15]{Caviglia2022}.  Thus some $\ell_s$ on the curve is
Green-good in degree $n$.  Torus invariance carries $\ell_s$ to
$\ell=\ell_1$ and preserves the quotient Hilbert function.  Since $n$ was
arbitrary (and only finitely many degrees occur), this fixed form is
Green-good in every degree.

Write $h_n=\dim B_n$, $q_n=\dim(B/\ell B)_n$, and $H=\sum h_n$. The degree-zero case is $q_0=h_0=1$.  For $n\geq1$, write the
$n$th Macaulay expansion of an integer $h\geq0$ as
\begin{equation*}
 h=\binom{a_n}{n}+\binom{a_{n-1}}{n-1}+\cdots+\binom{a_j}{j},
 \qquad
 a_n>a_{n-1}>\cdots>a_j\geq j,
\end{equation*}
and define its Green lowering by
\begin{equation*}
 h_{\langle n\rangle}
 :=
 \binom{a_n-1}{n}
 +\binom{a_{n-1}-1}{n-1}
 +\cdots+
 \binom{a_j-1}{j},
 \qquad
 0_{\langle n\rangle}:=0.
\end{equation*}
Since $\ell$ is Green-good, Green's hyperplane restriction theorem gives
\begin{equation*}
 q_n\leq(h_n)_{\langle n\rangle}.
\end{equation*}
For $e\geq2$, apply the after-scaling form of Green's theorem to a
lex quotient with degree-$n$ Hilbert function $h$.  Since Green's
estimate is sharp for lex quotients, Theorem~2.7 and
Corollary~3.19 of \cite{Greco2015} give
\begin{equation*}
    \frac{h_{\langle n\rangle}}
         {\binom{n+e-2}{e-2}}
    \leq
    \frac{h}
         {\binom{n+e-1}{e-1}}.
\end{equation*}
Equivalently,
\begin{equation*}
    h_{\langle n\rangle}
    \leq
    \frac{e-1}{n+e-1}h,
    \quad \text{ for }
    0\leq h\leq\binom{n+e-1}{e-1}.
\end{equation*}
Therefore
\begin{equation}\label{eq:green-degree-bounds}
 q_n\le \frac{e-1}{n+e-1}h_n \quad \text{and}\quad
 q_n\le\binom{n+e-2}{e-2},
\end{equation}
where the second inequality only uses that $B/\ell B$ is generated by $e-1$
linear variables.  Splitting the sum at
$T=\max\{1,\lceil H^{1/e}\rceil\}$ yields
\begin{align*}
 \sum_nq_n&\le \sum_{n=0}^{T-1}\binom{n+e-2}{e-2}+\frac{e-1}{T+e-1}\sum_{n\ge T}h_n\le\binom{T+e-2}{e-1}+\frac{e-1}{T+e-1}H
 \le C_eH^{(e-1)/e}.
\end{align*}
For every cyclic factor, including the cases $e_j\leq1$ treated above,
we have
\begin{equation*}
    \dim_{\K}B_j/\ell B_j
    \leq
    C_d H_j^{(d-1)/d},
    \quad
    H_j:=\dim_{\K}B_j.
\end{equation*}
Since there are at most $r$ nonzero cyclic factors, concavity gives
\begin{equation*}
    \sum_jH_j^{(d-1)/d}
    \leq
    r^{1/d}
    \left(\sum_jH_j\right)^{(d-1)/d}.
\end{equation*}
Apply this in \eqref{eq:green-cyclic-decomposition}, and together with
\eqref{eq:green-filtered-bound}, proves
\eqref{eq:weighted-green-boundary}.
\end{proof}

\begin{proposition}
\label{prop:norm-defect}
For every homogeneous quotient $A=\K[V_1,\ldots,V_d]/J$ and every
$t\ge0$,
\begin{equation}\label{eq:norm-defect}
        k_A(t)\le C_dH_A(t)^{(d-1)/d}.
\end{equation}
\end{proposition}

\begin{proof}
Let
\begin{equation*}
    \mathfrak{m}=(V_1,\ldots,V_d)
\end{equation*}
and set
\begin{equation*}
    A':=A/\mathfrak{m}^{t+1}A,
    \quad
    \mathcal K:=
    \left(0:_{A'[z]}\widetilde{\mathcal R}_d\right).
\end{equation*}
The quotient map $A\to A'$ is an isomorphism through degree $t$.
It therefore does not change the degree-$t$ source of multiplication by
$\widetilde{\mathcal R}_d$, although it may introduce additional kernel
elements. Hence
\begin{equation*}
    k_A(t)
    \leq
    \dim_{\K}\mathcal K_t.
\end{equation*}
The module $A'[z]$ is finite free over $\K[z]$. Its graded submodule
$\mathcal K$ is torsion free and hence free over the principal ideal
domain $\K[z]$. Since its homogeneous generators have nonnegative
degrees,
\begin{equation*}
    \dim_{\K}\mathcal K_t
    \leq
    \operatorname{rank}_{\K[z]}\mathcal K.
\end{equation*}
Consequently,
\begin{equation*}
    k_A(t)
    \leq
    \dim_{\K}\mathcal K_t
    \leq
    \operatorname{rank}_{\K[z]}\mathcal K.
\end{equation*}

Set
\[
 p_i=\sum_{j=1}^dV_j^{2i},\qquad 1\le i\le d.
\]
Consider the signed-permutation group, denoted by $W=(\mathbb Z/2\mathbb Z)^d\rtimes S_d$, which changes signs and
permutes the variables $V_i$. Since $p_1,\ldots,p_d$ are basic invariants for the signed-permutation
group, $\mathbb K[V]$ is free of rank $2^d d!$ over
$\mathbb K[p_1,\ldots,p_d]$. Hence $A'$ is generated by at most
$2^d d!$ elements over this ring.
Scalar dilation acts with the distinct
weights $2,4,\ldots,2d$, so $A'$ satisfies the equivariance and bounded
generator hypotheses of Theorem~\ref{thm:green-boundary}.  In the completed
local ring, all conjugate factors of the norm except the harmonic one are
units, while the remaining factor is
\begin{equation}\label{eq:norm-linearization}
 \sum_{j=1}^d\sqrt{4+V_j^2}-2d
   =\sum_{i=1}^da_ip_i+O((p_1,\ldots,p_d)^2),
 \qquad a_i\ne0
\end{equation}
Here we absorb the higher power sums into the error term by Newton identities.
The homogeneous ideal defining $A'$ is preserved by the invertible change
$V=zW$.  Over $\K(z)$ we therefore have the exact chain
\begin{equation}\label{eq:norm-rank-cokernel}
 \operatorname{rank}_{\K[z]}(0:\widetilde{\mathcal R}_d)
 =\dim_{\K(z)}\ker(\mathcal R_d:A'_{\K(z)}\to A'_{\K(z)})
 =\dim_{\K(z)}A'_{\K(z)}/\mathcal R_dA'_{\K(z)}.
\end{equation}
Since $(V_1,\ldots,V_d)^{t+1}A'=0$, each formal square root
\begin{equation*}
    \omega_i=\sqrt{4+V_i^2}
    =2+\frac14V_i^2-\frac1{64}V_i^4+\cdots
\end{equation*}
truncates to a well-defined element of $A'$. The norm factorization
therefore holds directly in $A'$. Every factor except
\begin{equation*}
    \Gamma=\sum_{i=1}^d\omega_i-2d
\end{equation*}
has nonzero constant term and is consequently a unit. Hence
$\mathcal R_d=U\Gamma$ for some unit $U$, so multiplication by
$\mathcal R_d$ and multiplication by $\Gamma$ have the same kernel and
cokernel. Thus Theorem~\ref{thm:green-boundary}
bounds \eqref{eq:norm-rank-cokernel} by
\begin{equation*}
    C_d(\dim A')^{(d-1)/d}=C_dH_A(t)^{(d-1)/d}.
\end{equation*}
This proves~\eqref{eq:norm-defect}.
\end{proof}

\subsection{Proof of the uniform lower bound}

\begin{proof}[Proof of Theorem~\ref{thm-quantitative}]
It is enough first to consider $n$ above a dimensional constant.  Put
\[
 R=\left\lfloor\frac{n}{4}\right\rfloor,
 \qquad r=R,
 \qquad J=J_R(E),
 \qquad H(t)=H_J(t),
\]
where $E=\supp(u)\cap Q_n$.  For every $t$ in a fixed 
interval $t_d\le t\le R$ for some positive $t_d$ depending only on $d$, apply
Theorem~\ref{thm:full-provenance} with
\[
 a=b=t,\qquad L_t=\left\lfloor\frac{t}{d}\right\rfloor.
\]
All margin conditions \eqref{eq:margin} hold after enlarging $t_d$.  Since $L_t+1\ge t/d$,
\eqref{eq:full-provenance}
and Proposition~\ref{prop:norm-defect} give
\begin{equation*}
 t^d\le C_d\left[
 H(t)\bigl(H(t)-H(t-D_d)\bigr)
       +H(t)^{(2d-1)/d}\right].
\end{equation*}
Since
\[
 H(t)\bigl(H(t)-H(t-D_d)\bigr)
 \le H(t)^2-H(t-D_d)^2,
\]
we have
\begin{equation*}
    t^d\le C_d\left[
 H(t)^2-H(t-D_d)^2
       +H(t)^{(2d-1)/d}\right].
\end{equation*}
Sum it through $R$, separately in $D_d$ residue classes, we obtain
\begin{equation*}
 R^{d+1}\le C_d\left[
       H(R)^2+R H(R)^{(2d-1)/d}\right].
\end{equation*}
Whichever term on the right dominates, we obtain
\[
 H(R)\ge c_dR^{\min\{(d+1)/2,\,d^2/(2d-1)\}}
       =c_dR^{d^2/(2d-1)}.
\]
Finally,
Lemma~\ref{lem:footprint} gives $H(R)\le|E|$.  For the finitely many
smaller radii use $|E|\ge1$ and decrease $c_d$.  The resulting constant
depends only on $d$ and works for arbitrary $n\ge1$.
\end{proof}

\section{Sharp quadratic lower bound for \texorpdfstring{$d=3$}{}}\label{sec:proof-sharp-3}

Our aim is to prove the remaining case $d=3$ of
Theorem~\ref{thm-quantitative}. Several auxiliary results below hold
for $d\geq2$ and will also be used in Section~\ref{sec:zariski}. By
Remark~\ref{rmk:equivalence}, it suffices in the quantitative argument
to consider a real-valued harmonic function $u$ with $u(0)\neq0$.

\subsection{Moment relations}
Let $d\ge 2$ and $u:\Z^d\to \R$ be real harmonic with $u(0)\neq 0$. Define the ball with $\ell^1$ radius $R$ as
\begin{equation*}
    B_R=\lbrace x\in \Z^d: \|x\|_1\le R\rbrace.
\end{equation*}

Using the shifts $T_i$ introduced in Section~\ref{sec:proof-any-d}, write
\begin{equation*}
    T^{-y}= T_1^{-y_1}\cdots T_d^{-y_d}
\end{equation*}
and let $[T^{-y}]F(T)$ denote the coefficient of the monomial $T^{-y}$ in $F$.

For $R\ge 0$ and $y\in \Z^d$, define
\begin{equation}\label{eqn:coefficient-notation}
    c_R(y):=[T^{-y}]P(T)^R.
\end{equation}

The following lemma gives a degree lowering property for polynomial multiples of a harmonic function.

\begin{lemma}\label{lem:lower}
    Let $q$ be a polynomial in $\C[x_1,\cdots,x_d]$ of total degree $m$. Then
    \begin{equation*}
        P^{m+1}(qu)=0.
    \end{equation*}
\end{lemma}
\begin{proof}
Let $D(A)=[P,A]$. As in the proof of
Lemma~\ref{lem:commutator}, each application of $D$ lowers the degree
of the polynomial multiplier by one. Hence
\begin{equation*}
    D^{m+1}(M_q)=0.
\end{equation*}
Since $Pu=0$, the iterated-commutator identity gives
\begin{equation*}
    P^{m+1}(qu)=D^{m+1}(M_q)u=0.
\end{equation*}
\end{proof}

\subsection{Independent centered relations}
We reformulate a form of the Cayley--Bacharach property, see~\cite[Definition~2.1]{GKR} or~\cite[\S~1.2]{EGH} for a discussion about its history.
\begin{definition}
    Let $Z\subset \C^d$ be a finite set and let $s\ge 0$ be an integer. We say that $Z$ satisfies $CB(s)$ if for every $z\in Z$ and every polynomial $q\in \C[x_1,\cdots,x_d]$ of degree at most $s$, one always has
    \begin{equation*}
         q|_{Z\setminus\{z\}}=0 \quad\Longrightarrow\quad q(z)=0.
    \end{equation*}
\end{definition}

Fix $a\in \Z^d, R\ge 1$, and set
\begin{equation*}
    E(a,R):=\supp (u)\cap (a+B_R).
\end{equation*}
The following lemma shows that $E(a,R)$ satisfies $CB(R-1)$.
\begin{lemma}\label{lem:cb}
    There are $w_x\neq 0$ for any $x\in E(a,R)$ such that for evey polynomial with degree smaller than $R$, we have
    \begin{equation}\label{eq:moment-relation}
        \sum_{x\in E(a,R)}w_xq(x)=0.
    \end{equation}
    Consequently $E(a,R)$ satisfies $CB(R-1)$.
\end{lemma}
\begin{proof}
    By Lemma~\ref{lem:lower}, we obtain
    \begin{equation*}
        0=(P^R(qu))(a)=\sum_{x\in a+B_R}c_R(a-x)u(x)q(x),
    \end{equation*}
    where the last identity used a translation $y\mapsto a-y$. 

    For $y\in B_R$, expanding the definition \eqref{eqn:coefficient-notation} by the explicit form of $P$ gives
    \begin{equation*}
        c_R(y)=\sum_{j=0}^R\binom Rj(-2d)^{R-j}[T^{-y}]A(T)^j\quad \text{with } A(T)=\sum_{i=1}^d(T_i+T_i^{-1}).
    \end{equation*}
    The last coefficient counts the number of ways from $0$ to $-y$ with $j$ steps. Whenever it is nonzero, we always have $j\equiv \|y\|_1 (\mod 2)$. Hence all nonzero summands have the same sign $(-1)^{R-\|y\|_1}$. For $j = \|y\|_1\le R$, there must exists a shortest path, therefore the term $c_R(y)\neq 0$. This implies $w_x=c_R(a-x)u(x)\neq 0$ since $x\in E(a,R)$. This proves the fully nonvanishing coefficients in \eqref{eq:moment-relation}. Then $E(a,R)$ satisfies $CB(R-1)$ since there cannot be only one nonzero term survive in \eqref{eq:moment-relation}.
\end{proof}

The following lemma shows that the relations \eqref{eq:moment-relation} are linearly independent. We use the evaluation Hilbert function $H_E(s)$ from Definition~\ref{def-new-hilbert} with $\K=\C$.
\begin{lemma}\label{lem:gram}
    Let $D$ be a positive even integer and set
$s:=D/2$. Let $\Omega\subset \Z^d$ be finite and set $E=\supp(u)\cap\Omega$. Assume that a set $F\subset E$ has $a+B_D$ inside $\Omega$ for every $a\in F$. For any $a\in F$ and $x\in E$, define
    \begin{equation*}
        L_a(x)=c_D(a-x)u(x).
    \end{equation*}
    Then the vectors $L_a\in \R^{E}$, $a\in F$, are linearly independent degree-$D-1$ evaluation relations. Hence
    \begin{equation}\label{eq:rank-nulity}
        H_E(D-1)\le |E|-|F|.
    \end{equation}
\end{lemma}
\begin{proof}
We prove linear independence by contradiction. Assume linear dependence relation
\begin{equation}\label{eq:linear-relation}
    \sum_{a\in F}\alpha_a L_a=0
\end{equation}
for some $\alpha_a$ being nonzero.
    Set $p_s(z)=[T^z]P(T)^s$. Symmetry gives
    \begin{equation*}
        c_D(a-b)=\sum_{z}p_s(z-a)p_s(z-b).
    \end{equation*}
    This Gram matrix $G=(c_D(a-b))_{a,b\in F}$ is strictly positive definite. Indeed, zero Gram norm would give
    \begin{equation*}
         \sum_{a,b\in F}\alpha_a\alpha_b c_D(a-b)
 =\sum_z\left(\sum_{a\in F}\alpha_ap_s(z-a)\right)^2=0
    \end{equation*}
    and therefore
    \begin{equation*}
        P(T)^s\sum_{a\in F}\alpha_aT^a=\left(\sum_{w\in \Z^d}p_s(w)T^w\right)\left(\sum_{a\in F}\alpha_aT^a\right)
 =
 \sum_{z\in\Z^d}
 \left(\sum_{a\in F}\alpha_ap_s(z-a)\right)T^z=0.
    \end{equation*}
    It is impossible for nonzero $(\alpha_a)$ since it is in a Laurent polynomial domain. Divide \eqref{eq:linear-relation} by $u(x)\neq 0$ on $E$ shows that $g(x)=\sum_{a\in F}\alpha_a c_D(a-x)$ vanishes on $E$, hence on $F$. Thus the Gram matrix kills $\alpha$ and makes a contradiction. This proves linear independence.
    
    By definition, the coefficient $c_D(a-x)$ vanishes unless $x\in a+B_D$. Since $a+B_D\subset \Omega$, Lemma~\ref{lem:cb} gives 
    \begin{equation*}
        \sum_{x\in E}L_a(x)q(x)=0, \quad \forall \deg q \le D-1.
    \end{equation*}
    Thus we obain the desired rank inequality~\eqref{eq:rank-nulity}.
\end{proof}

\begin{corollary}\label{crc:gram}
    If $r\ge D$ and $E=\supp (u)\cap B_r$, then
    \begin{equation*}
        H_{E}(D-1)\le |\supp(u)\cap (B_r\setminus B_{r-D})|.
    \end{equation*}
\end{corollary}
\begin{proof}
    Use $F=\supp(u)\cap B_{r-D}$ in Lemma~\ref{lem:gram}.
\end{proof}

\subsection{Flat Hilbert growth yields a full curve}

\begin{lemma}\label{lem:carrier}
    Let $m\ge 1$ be a real number, let $D$ be even with 
    \begin{equation*}
        10^4 m \le D\le 10^4 m +2.
    \end{equation*}
    Suppose finite nonempty set $Z\subset \C^d$ has $CB(t)$ for some $t\ge D$ and
    \begin{equation*}
        H_Z(D-1)\le 20 Dm.
    \end{equation*}
    Then $Z$ lies on a reduced curve of degree at most $160m$.
\end{lemma}
\begin{proof}
    Let $h_j=H_Z(j)-H_Z(j-1)$ and set $N=D-1$. 
    
    First we claim that there exists some $j_0\in \left[\lceil N/3\rceil, \lfloor 2N/3\rfloor\right]$ such that
    \begin{equation*}
        h_{j_0}\le 160m<j_0.
    \end{equation*}
    Indeed, the interval has at least $N/4$ terms and $\sum_{j\le N}h_j\le 20 Dm$ with $D/N<2$, which gives $h_{j_0}\le 160m$. 

    Macaulay's growth inequality \cite[Chapter~4]{Bruns_Herzog} makes $h_j$ nonincreasing for $j\ge j_0$: if $h_j=a\le j$, its Maculay sucessor is again $a$. Therefore
    \begin{equation*}
        h_{j+1}\le h_{j}, \forall j\ge j_0.
    \end{equation*}
    The conditions $CB(t)$ and $t\ge D$ implies
    \begin{equation*}
        H_Z(D)<|Z|.
    \end{equation*}
     This implies $h_j$ is positive through $D$. Since $D-j_0\ge D/3 -1>160m$, there are more than $160m$ transitions after $j_0$. Thus some $j_0\le s<D$ satisfies
     \begin{equation*}
         h_s=h_{s+1}=e>0,\quad e\le 160m<s.
     \end{equation*}

     The theorem of Bigatti--Geramita--Migliore \cite{BGM}, see also \cite[Theorem~5.3(a)]{Migliore}, supplies $V\supset Z$. Let $C$ be the union of the one-dimensional components of $V$. If $z\in Z\setminus C$, then reducedness makes $\lbrace z\rbrace$ an open-and-closed component of $V$. Its idempotent in $H^0(V,\mathcal{O}_V)$ multiplied by $L^s$ for a linear form with $L(z)\neq 0$, is a section of $\mathcal{O}_V(s)$ which is nonzero at $z$ and zero on every other component. Since $\mathcal{I}_V$ is $s$-regular, we have $H^1(\mathbb{P}^d,\mathcal{I}_V(s))=0$, and therefore the surjection
     \begin{equation*}
         H^0(\mathbb{P}^d,\mathcal{O}(s))\twoheadrightarrow H^0(V,\mathcal{O}_V(s)).
     \end{equation*}
     Then the section therefore lifts to a degree-$s$ ambient form. This form separates $z$ from $Z\setminus \lbrace z\rbrace$, which is contrrary to $CB(t)$ since $s<D\le t$. Hence $Z\subset C$ and $\deg C=e\le 160m$.
\end{proof}

\subsection{No low degree curve carriers}

\begin{lemma}\label{lem:interp}
Let $C\subset\mathbb A^d_{\C}$ be an irreducible affine curve of degree
$e$, let
$\varnothing\ne Z\subset C$ be finite and nonempty, and let $q_0\ge 0$. Suppose there are coefficients
$c_z\ne0$ for every $z\in Z$ such that
\begin{equation*}
 \sum_{z\in Z}c_zq(z)=0
 \quad \forall q\in \C[x_1,\cdots,x_d] \text{ with }\deg q\le q_0.
\end{equation*}
Then
\begin{equation}\label{eq:curve-interpolation}
 |Z|\ge e(q_0-e+3).
\end{equation}
\end{lemma}

\begin{proof}
    Project the projective closure generically and birationally to an integral plane curve $\Gamma$ of degree $e$, with the projection injective on $Z$. If \eqref{eq:curve-interpolation} failed, then
    \begin{equation*}
        \deg \mathcal{I}_{Z,\Gamma}(q_0)=eq_0-|Z|>e(e-3)=\deg \omega_\Gamma.
    \end{equation*}
    Hence there is no nonzero homomorphism from the rank-one torsion-free sheaf on the left to the one on the right. Serre duality \cite[Chapter III, Theorem 7.6]{Hartshorne1977} gives $H^1(\Gamma, \mathcal{I}_{Z,\Gamma}(q_0))=0$, so we obtain the surjection
    \begin{equation*}
        H^0(\Gamma,\mathcal{O}_\Gamma(q_0))\twoheadrightarrow  H^0\bigl(Z,\mathcal{O}_Z(q_0)\bigr)
\cong \bigoplus_{z\in Z}\mathbb{C}.
    \end{equation*}
    Since $H^1(\mathbb{P}^2,\mathcal{O}_{\mathbb{P}^2}(q_0-e))=0$, then we obtain
    \begin{equation*}
        H^0(\mathbb{P}^2,\mathcal{O}_{\mathbb{P}^2}(q_0))\twoheadrightarrow H^0(\Gamma, \mathcal{O}_\Gamma(q_0)).
    \end{equation*}
    Composing this with the preceding surjection, 
    one may pull back a polynomial $q$ satisfying $$\deg q \le q_0, q(z_0)=1,q(z)=0,\forall z\in Z\setminus\lbrace z_0 \rbrace$$ contradicts the displayed dependence.
\end{proof}

\begin{lemma}\label{lem:packets}
    Let $C=\bigcup_{\alpha=1}^sC_\alpha$ be reduced with $\deg C_\alpha =e_\alpha$ and $\delta=\sum_{\alpha=1}^s e_\alpha$. Let $E\subset C$ be finite, and suppose that every component $C_\alpha$ contains a point of $E$ lying on no other component. Define
    \begin{equation*}
        E_\alpha:=(E\cap C_\alpha)\setminus \bigcup_{\beta\neq \alpha}C_\beta.
    \end{equation*}
    If there exist coefficients $c_x\neq 0$ for every $x\in E$ such that
    \begin{equation*}
        \sum_{x\in E}c_xq(x)=0,\quad \forall q\in \C[x_1,\cdots,x_d] \text{ with }\deg q\le r-1,
    \end{equation*}
    then for every $\alpha$, we have
    \begin{equation*}
        |E_\alpha|\ge e_\alpha (r+2-\delta).
    \end{equation*}
\end{lemma}
\begin{proof}
    Fix $\alpha$. For each $\beta\neq \alpha$, a generic projection of $\overline{C}_\beta$ to a plane gives a cone hypersurface of degree $e_\beta$ containg $C_\beta$, avoiding every point of $E\setminus C_\beta$ and not containing $C_\alpha$. To see that this remains possible in $\mathbb{P}^d$, one may choose a center $\mathbb{P}^{d-3}$ in the open locus of birational projections and away from the finitely many joins of those points with $\overline{C}_\beta$. Every such join has dimension at most two, and $(d-3)+2<d$, hence a general center avoids all of them.

    The product $G_\alpha$ of these hypersurface equations has degree at most $\delta-e_\alpha$, and its nonzero locus on $E$ is exactly $E_\alpha$. Multiplying test polynomials by $G_\alpha$ gives
    \begin{equation*}
        \sum_{x\in E_\alpha}c_xG_\alpha(x)q(x)=0, \quad \forall q\in \C[x_1,\cdots,x_d] \text{ with }\deg q\le q_0,
    \end{equation*}
    where $q_0=r-1-\delta+e_\alpha$. For $q_0\ge 0$, since $c_xG_\alpha(x)\neq 0$ for all $x\in E_\alpha$, by Lemma~\ref{lem:interp} we obtain
    \begin{equation*}
        |E_\alpha|\ge e_\alpha (q_0-e_\alpha+3)=e_\alpha(r+2-\delta).
    \end{equation*} 
    For $q_0<0$, then $r\le \delta-e_\alpha$. Since $e_\alpha\ge 1$ and $E_\alpha\neq \varnothing$, we obtain
    \begin{equation*}
        e_\alpha(r+2-\delta)\le e_\alpha(2-e_\alpha)\le 1\le |E_\alpha|.
    \end{equation*}
    Since $\alpha$ is arbitrary, the proof is complete.
\end{proof}

We introduce the forward port displacements for $i\in \lbrace 1,2,\cdots,d\rbrace$ as follows
\begin{equation*}
    \mathcal{D}_i=\lbrace e_i,2e_i\rbrace \cup \lbrace e_i+e_j,e_i-e_j;j\neq i\rbrace.
\end{equation*}

\begin{lemma}\label{lem:ports}
    Let $A_d u  =\lambda u$. For any $p\in \supp (u)$ and any coordinate index $i$, there exists at least one $h\in \mathcal{D}_i$ satisfies $p+h\in\supp(u)$, and the displacement $h=(h_1,\cdots,h_d)$ obeys
    \begin{equation}\label{eq:port-cone}
        h_i>0. \quad \sum_{j\neq i}|h_j|\le h_i.
    \end{equation}
\end{lemma}
\begin{proof}
    If $u(p+e_i)\neq 0$, then choose $h=e_i$ and the proof is finished. If not, we evaluate the eigenvalue equation at $p+e_i$
    \begin{equation*}
        0=\lambda u (p+e_i)=\sum_{y\sim p+e_i} u(y).
    \end{equation*}
    There already exists one term $u(p)\neq 0$ in the sum. Therefore there must be another supported point in $p+\mathcal{D}_i$. The inequality \eqref{eq:port-cone} is immediate and we finish the proof.
\end{proof}

Then we give the following dimension estimate of subvarieties satsifying translation symmetry.

\begin{proposition}\label{prop:port-cycle}
    Let $\mathcal{C}$ be a finite nonempty family of $k$-dimensional irreducible affine subvarieties of $\mathbb A^d_{\C}$. Suppose that for any $V\in \mathcal{C}$ and any $i\in \lbrace 1,2,\cdots,d\rbrace$, there exist $V'\in \mathcal{C}$ and $h\in \mathcal{D}_i$ such that
    \begin{equation}\label{eq:port-cycle-edge}
        V'=V+h.
    \end{equation}
    Fix $V\in \mathcal{C}$, then there exist vectors $v_1,\cdots,v_d\in \Z^d$ satisfying
    \begin{equation}\label{eq:port-cycle-stabilizers}
        V+v_i=V, \quad (v_i)_i>0, \quad \sum_{j\neq i}|(v_i)_j|\le (v_i)_i.
    \end{equation}
    Set 
    \begin{equation*}
        W=\mathrm{span}_\C \lbrace v_1,\cdots,v_d \rbrace, \quad a_i=\frac{v_i}{(v_i)_i},
    \end{equation*}
    and let $M$ be the real $d\times d$ matrix whose $i$-th row is $a_i$. Then we have
    \begin{equation}\label{eq:port-cycle-invariance}
        V+w=V,\quad \forall w\in W,
    \end{equation}
    \begin{equation}\label{eq:port-cycle-dimension}
        k\ge \dim W = \operatorname{rank}M\ge \left\lceil \frac{d}{2}\right\rceil,
    \end{equation}
    and
    \begin{equation}\label{eq:port-cycle-matrix-bounds}
        \operatorname{tr}M=d, \quad \|M\|^2\le 2d.
    \end{equation}
\end{proposition}

\begin{proof}
    Fix $V\in \mathcal{C}$. We construct the vector $v_i$ for each coordinate index $i$.

    For every member of $\mathcal{C}$, choose one outgoing edge of color $i$ as in \eqref{eq:port-cycle-edge}. Starting from $V_0:=V$ and following the chosen edges one by one gives 
    \begin{equation}\label{eq:component-walk}
        V_{\nu+1}=V_{\nu}+h_\nu,\quad h_\nu \in \mathcal{D}_i.
    \end{equation}
    Since the family $\mathcal{C}$ is finite, there must be some $0\le a<b$ such that $V_a=V_b$. Define
    \begin{equation*}
        g_i=\sum_{\nu=0}^{a-1}h_\nu \quad  \text{and}\quad v_i=\sum_{\nu=a}^{b-1}h_\nu,
    \end{equation*}
    where $g_i=0$ if $a=0$. By \eqref{eq:component-walk} we have
    \begin{equation*}
        V_a=V+g_i \quad \text{and}\quad V_b=V_a+v_i.
    \end{equation*}
    Since $V_a=V_b$, we have $(V+g_i)+v_i=V+g_i$.
Translating by $-g_i$ yields
\begin{equation}\label{eq:appendix-cycle-stabilizes-start}
    V+v_i=V.
\end{equation}
By Lemma~\ref{lem:ports}, every $h_\nu \in \mathcal{D}_i$ satisfies \eqref{eq:port-cone}. Hence
\begin{equation*}
    (v_i)_i=\sum_{\nu=a}^{b-1}(h_\nu)_i>0
\end{equation*}
and
\begin{equation*}
    \sum_{j\neq i}|(v_i)_j|\le \sum_{\nu=a}^{b-1}\sum_{j\neq i}|(h_\nu)_j|\le \sum_{\nu=a}^{b-1}(h_\nu)_i=(v_i)_i.
\end{equation*}
Thus $v_i\in \Z^d$ satisfies \eqref{eq:port-cycle-stabilizers}. 

Let $F$ be a polynomial vanishing on $V$ and fix $x\in V$. By \eqref{eq:appendix-cycle-stabilizes-start}, the set of points $\lbrace x+nv_i:n\in\Z\rbrace$ is contained in $V$ and therefore the one-variable polynomial
\begin{equation*}
    z\longmapsto F(x+zv_i)
\end{equation*}
vanishes at every integer. This imples that $x+zv_i\in V$ for any $z\in \C$. Since $x$ is arbitrary, we obtain $V+zv_i\subset V$ for every $z\in \C$. Thus
\begin{equation*}
    V+zv_i=V,\quad \forall z\in \C.
\end{equation*}
Combining these translation invariances for $i=1,\cdots,d$ gives \eqref{eq:port-cycle-invariance}. In particular, $x+W\subset V$ for any $x\in V$, and this implies 
\begin{equation}\label{eq:appendix-W-dimension}
    \dim W\le \dim V=k.
\end{equation}

The normalization $a_i=v_i/(v_i)_i$ makes the diagonal entry equal one, hence the last inequality in \eqref{eq:port-cycle-stabilizers} gives
\begin{equation*}
    \sum_{j\neq i}|M_{ij}|\le 1.
\end{equation*}
Consequently, we obtain $\operatorname{tr}M=d$ and
\begin{equation}\label{eq:appendix-matrix-frobenius}
    \|M\|^2=\sum_{i=1}^d\left(1+\sum_{j\neq i}|M_{ij}|^2\right)\le \sum_{i=1}^d
    \left(
        1+
        \left(\sum_{j\neq i}|M_{ij}|\right)^2
    \right)
    \leq 2d.
\end{equation}
This proves \eqref{eq:port-cycle-matrix-bounds}. 

Let $\sigma_1,\cdots,\sigma_{r_M}$ be all nonzero singular values of $M$, where $r_M:=\operatorname{rank}M$. The trace bound by the nuclear norm, followed by Cauchy--Schwarz inequality and \eqref{eq:appendix-matrix-frobenius}, gives
\begin{equation*}
    d=|\operatorname{tr}M|\le \sum_{\nu=1}^{r_M}\sigma_\nu \le \sqrt{r_M}\left(\sum_{\nu=1}^{r_M}\sigma_\nu^2\right)^{1/2}=\sqrt{r_M}\|M\|\le \sqrt{2dr_M}.
\end{equation*}
Hence $r_M\ge d/2$, and integrality gives $r_M\ge \lceil d/2\rceil$. This combines with \eqref{eq:appendix-W-dimension} gives \eqref{eq:port-cycle-dimension} and the proof is complete.
\end{proof}

\begin{proposition}\label{prop:nocurve}
Let $d\ge 3$. Let $r\ge1$ be an integer, let $E=\supp(u)\cap B_r$, and suppose
$u(0)\ne0$. The set $E$ cannot be contained in a reduced curve of degree
$\delta$ when
\begin{equation}\label{eq:no-curve-hypotheses}
 \delta<\frac r{8d},
 \qquad
 |\supp(u)\cap(B_r\setminus B_{r-2})|<\frac r4.
\end{equation}
\end{proposition}
\begin{proof}
    Suppose to the contrary, that $E$ is contained in a reduced curve of degree $\delta$ satisfying \eqref{eq:no-curve-hypotheses}. After discarding components at infinity and then redundant affine components, we may write the remaining curve as $\bigcup_\alpha C_\alpha$ where each $C_\alpha$ contains a point of $E$ lying on no other component. Let $e_\alpha=\deg C_\alpha$ and $\delta=\sum_{\alpha}e_\alpha$. Lemmas~\ref{lem:cb} and~\ref{lem:packets} give
    \begin{equation*}
        |E_\alpha|\ge e_\alpha (r+2-\delta)> \frac{3r}{4}e_\alpha.
    \end{equation*}
    After removing the common boundary shell, the interior packet $\Omega_\alpha:=E_\alpha\cap B_{r-2}$ satisfies
    \begin{equation}\label{eq:omega-bound}
        |\Omega_\alpha|>\frac{r}{2}e_\alpha.
    \end{equation}

    Fix $i$. For every $p\in \Omega_\alpha$, choose by Lemma~\ref{lem:ports} a supported $p+h$, $h\in \mathcal{D}_i$, and a component $C_\beta$ containing it. Since $\|h\|_1\le 2$, the successor is in $B_r$ and
    \begin{equation*}
        p\in C_\alpha\cap (C_\beta-h).
    \end{equation*}
    If $C_\alpha$ and $C_\beta-h$ are distinct, we choose a generic projection cone of degree $e_\beta$ containing $C_\beta-h$ but not $C_\alpha$. Bézout's Theorem \cite[Chapter~I, Theorem~7.7]{Hartshorne1977} on $C_\alpha$ then bounds their intersection by $e_\alpha e_\beta$. Summing over $2d$ displacements and all $\beta$, distinct pairs account for at most
    \begin{equation*}
        2d e_\alpha \delta<\frac{r}{4}e_\alpha
    \end{equation*}
    points. Comparison with \eqref{eq:omega-bound} forces a translation edge
    \begin{equation*}
        C_\beta=C_\alpha+h, \quad h\in \mathcal{D}_i
    \end{equation*}
    for every $C_\alpha$ and every $i$.

    Hence the finite family $\lbrace C_\alpha\rbrace$ satisfies the hypothesis of Proposition~\ref{prop:port-cycle} with $k=1$. That proposition gives $1\ge \lceil d/2\rceil$. It is important for $d\ge 3$, therefore the assumed curve carrier does not exist.
\end{proof}

\subsection{Proof of the quadratic lower bound}

First we establish a linear lower bound.
\begin{lemma}\label{lem:ray}
If $u(0)\ne0$, then
\begin{equation}\label{eq:ray-bound}
 |\supp(u)\cap B_n|\ge n+1,\quad \forall n\ge0.
\end{equation}
\end{lemma}
\begin{proof}
    Without loss of generality, assume $u(0)>0$. Let $K$ be the connected component of $0$ in the induced subgraph on $\lbrace u>0\rbrace$. If $K$ were finite, summing
    \begin{equation*}
        2du(x)-\sum_{y\sim x}u(y)=0
    \end{equation*}
    over $K$ gives
    \begin{align*}
 0
 &=\sum_{x\in K}\bigl(2d-\deg_K(x)\bigr)u(x)
   -\sum_{x\in K}\sum_{\substack{y\notin K\\y\sim x}}u(y).
\end{align*}
Every $u(x)$ with $x\in K$ is positive. At least one vertex of the finite
set $K$ has a neighbor outside $K$, so the first sum is strictly positive.
If $x\in K$, $y\notin K$, and $y\sim x$, then $u(y)\le0$. Thus the second term, with
its displayed minus sign, is nonnegative. Then the right-hand side is strictly positive, a contradiction. Hence $K$ is infinite. If locally finite infinite connected graph has a ray from $0$, then its first $n+1$ vertices lie in $B_n$, which completes the proof.
\end{proof}

\medskip
\begin{proof}[Proof of quadratic lower bound for $d\ge 3$]
    Let 
    \begin{equation*}
        M=|\supp(u)\cap B_n|,
 \quad m=\frac Mn,
 \quad c_d=\frac{10^{-10}}d.
    \end{equation*}
    Lemma~\ref{lem:ray} gives $m\ge 1$. Suppose for a contradiction that $M<c_d n^2$ for some $n$ and set
    \begin{equation*}
        D=2\lceil 5000m\rceil.
    \end{equation*}
    Then $10^4m\le D<10^4m+2$. The bound~\eqref{eq:ray-bound} and the contradiction assumption $M<c_dn^2$ gives $n+1<c_d n^2$, and hence $nc_d>1$. In particular, we have
    \begin{equation*}
        D<10002c_d n<n/100.
    \end{equation*}
     Each support point belongs to at most $D$ shells $B_r\setminus B_{r-D}$ and at most two shells $B_r\setminus B_{r-2}$. Then
\begin{align*}
 &\sum_{r\in [n/3,2n/3]}\bigl(
 |\supp(u)\cap(B_r\setminus B_{r-D})|
 +D|\supp(u)\cap(B_r\setminus B_{r-2})|\bigr)\\
 &\qquad\le(D+2D)M\le3DM
\end{align*}
where we used the fact that there are at least $n/4$ numbers in $[n/3,2n/3]$.
Dividing by $n/4$, we find an $r\in [n/3,2n/3]$ such that
\begin{align}\label{eq:thin-shell-choice}
 &|\supp(u)\cap(B_r\setminus B_{r-D})|
 +D|\supp(u)\cap(B_r\setminus B_{r-2})|
 \le12Dm.
\end{align}
Put $E=\supp(u)\cap B_r$. Corollary~\ref{crc:gram} and~\eqref{eq:thin-shell-choice} give $H_E(D-1)\le 12 Dm$. Lemma~\ref{lem:carrier} therefore places $E$ on a reduced curve of degree
\begin{equation}\label{eq:degree-bound}
    \delta \le 160m<160c_dn.
\end{equation}
The second term in \eqref{eq:thin-shell-choice} also gives
\begin{equation*}
    |\supp(u)\cap(B_r\setminus B_{r-2})|\le12m.
\end{equation*}
Since $r\ge n/3$ and $c_d=10^{-10}/d$, the degree bound~\eqref{eq:degree-bound} is smaller than $r/(8d)$ and $12m<r/4$. This contradicts Proposition~\ref{prop:nocurve}. Thus $|\supp(u)\cap Q_n^{(d)}|\ge |\supp(u)\cap B_n|\ge c_dn^2$.
\end{proof}

\section{Proof of Theorem \ref{thm-zariski}}\label{sec:zariski}
In this section, we prove Zariski-dimension lower bounds in Theorem \ref{thm-zariski}. See Subsection~\ref{subsec:sharp-zariski} for the sharpness.

For an eigenfunction $A_du=\lambda u$, we retain the notation
$P_\lambda=A_d-\lambda I$ from Section~\ref{sec:proof-any-d}. We also use the displacement
sets $\mathcal D_i$, Lemma~\ref{lem:ports}, and
Proposition~\ref{prop:port-cycle} from the preceding section.

\begin{proof}[Proof of Theorem \ref{thm-zariski}: easy case]
Throughout the whole proof, the topology we consider is Zariski topology.
    Let $S=\supp (u)$ and let 
    \begin{equation*}
        X:=\overline{S}^{\operatorname{Zar}}\subset \mathbb{A}_\C^d.
    \end{equation*}
    Let $k:=\dim X$ and $\mathcal{C}$ be the finite family of all irreducible components of $X$. Let $\mathcal{C}_k$ be the family of all irreducible components having the same dimension with $X$. 
    
    First, we verify that $\mathcal{C}_k$ satisfies the hypothesis of Proposition~\ref{prop:port-cycle}. Fix $V\in \mathcal{C}_k$ and set
    \begin{equation*}
        V^\circ:=V\setminus \bigcup_{V'\neq V,
        V'\in \mathcal{C}}V'.
    \end{equation*}
    Thus $V^\circ$ is a nonempty open subset of $X$ and dense in $V$. Since $S$ is dense in $X$, the set $S\cap V^\circ$ is dense in $V$. 

    Fix coordinate index $i$. For any $p\in S\cap V^\circ$, lemma~\ref{lem:ports} gives an $h_p\in \mathcal{D}_i$ such that $p+h_p\in S$. Let $V_p \in \mathcal{C}$ be an irreducible component $V_p$ of $X$ containing $p+h_p$. Then we have
    \begin{equation*}
        p\in V\cap (V_p-h_p).
    \end{equation*}
    Since $\mathcal{C}$ and $\mathcal{D}_i$ are finite, there are only finitely many possible pairs $(V',h)$. Then write
    \begin{equation*}
        S\cap V^\circ \subset \bigcup_{V'\in \mathcal{C}}\bigcup_{h\in \mathcal{D}_i}(V\cap (V'-h))\subset V.
    \end{equation*}
    The middle term is again a finite union of closed subsets of $V$ and contains a dense subset of $V$, so we obtain
    \begin{equation*}
         V = \bigcup_{V'\in \mathcal{C}}\bigcup_{h\in \mathcal{D}_i}(V\cap (V'-h)).
    \end{equation*}
    Hence there must be some irreducible component $V'$ and some $h\in \mathcal{D}_i$ such that $V\subset V'-h$
    Both varieties are irreducible, thus
    \begin{equation*}
        V'=V+h \quad \text{and}\quad V'\in \mathcal{C}_k.
    \end{equation*}    
    Since $V$ and $i$ are arbitrary, the hypothesis of Proposition~\ref{prop:port-cycle} holds for $\mathcal{C}_k$.

    Fix $V\in \mathcal{C}_k$. Apply Proposition \ref{prop:port-cycle} and let $W$ and $M$ be the invariant subspace and normalized matrix. Then we have
    \begin{equation*}
        V+w=V,\quad \forall w\in W
    \end{equation*}
    and
    \begin{equation*}
        k\ge \dim W=\operatorname{rank}M\ge \left\lceil \frac{d}{2}\right\rceil.
    \end{equation*}
    This gives the proof of \eqref{eq:zariski-general}. Moreover, if $d$ is odd, then 
    \begin{equation*}
        \left\lceil \frac{d}{2}\right\rceil=\left\lfloor \frac{d}{2}\right\rfloor+1.
    \end{equation*}
    Hence \eqref{eq:zariski-nonzero} is also proved for odd dimension case.
\end{proof}

Now the only remaining part is the even dimension case for $\lambda\neq 0$. 
We introduce the following lemma that will be used in the proof of additional rigidity in even dimensional case.
\begin{lemma}\label{lem:paired-layer-rigidity}
    Let $q\ge 1$ and $\lbrace a_j,b_j\rbrace_{j=1}^q$ be a partition of $\lbrace 1,2,\cdots,2q\rbrace$. Set
    \begin{equation*}
        W:=\operatorname{span}_\C\lbrace e_{a_j}+\epsilon_j e_{b_j}:1\le j\le q\rbrace 
    \end{equation*}
    for some sign vector $(\epsilon_1,\cdots,\epsilon_q)\in \lbrace \pm 1\rbrace^q$. Suppose that $w:\Z^{2q}\to \C$ is supported on a finite union of affine planes parallel to $W$ satisfying
    \begin{equation*}
        \dim_{\C}\overline{\supp(w)}^{\,\mathrm{Zar}}=q,\quad
 \dim_{\C}\overline{\supp((A_{2q}-\lambda)w)}^{\,\mathrm{Zar}}<q.
    \end{equation*}
    Then $\lambda=0$.
\end{lemma}

\begin{proof}
Set
\begin{equation*}
    t_j=x_{a_j},\qquad r_j=x_{b_j}-\epsilon_jx_{a_j}\qquad (1\le j\le q).
\end{equation*}
Thus the planes parallel to $W$ are the level sets of $r=(r_1,\cdots,r_q)$. For $r\in\Z^q$, write
\begin{equation*}
    w_r(t):=w(x(t,r)),\qquad x_{a_j}(t,r)=t_j,\quad x_{b_j}(t,r)=r_j+\epsilon_jt_j.
\end{equation*}
Only finitely many of the sequences $w_r$ are nonzero.

Let $\mathcal R=\C[T_1^{\pm1},\cdots,T_q^{\pm1}]$ act by shifts on $\mathscr S=\C^{\Z^q}$, where $(T_jf)(t)=f(t+e_j)$, and set
\begin{equation*}
    \mathscr N=\left\{f\in\mathscr S:\dim_{\C}\overline{\supp(f)}^{\,\mathrm{Zar}}<q\right\}.
\end{equation*}
The space $\mathscr N$ is an $\mathcal R$-submodule. Hence the finite Laurent polynomial
\begin{equation*}
    U(z):=\sum_{r\in\Z^q}[w_r]z^r
    \in(\mathscr S/\mathscr N)[z_1^{\pm1},\cdots,z_q^{\pm1}]
\end{equation*}
is nonzero. Indeed if not, every layer of $\supp(w)$ would have Zariski dimension less than $q$, and the finite-layer hypothesis would give the same conclusion for $\supp(w)$.

Set $g=(A_{2q}-\lambda)w$ and define $g_r$ similarly. If $Z=\overline{\supp(g)}^{\,\mathrm{Zar}}$, then $\supp(g_r)$ is contained in the inverse image of $Z$ under the affine embedding $t\mapsto x(t,r)$. Thus $g_r\in\mathscr N$ for every $r$. On the layer-generating function, the shifts by $\pm e_{a_j}$ contribute $T_jz_j^{\epsilon_j}+T_j^{-1}z_j^{-\epsilon_j}$, while those by $\pm e_{b_j}$ contribute $z_j+z_j^{-1}$. Consequently,
\begin{equation}\label{eq:paired-layer-module}
    P(z)U(z)=0,\qquad
    P(z):=\sum_{j=1}^q\left((1+T_j^{\epsilon_j})z_j+(1+T_j^{-\epsilon_j})z_j^{-1}\right)-\lambda.
\end{equation}
If $\lambda\ne0$, the constant coefficient $-\lambda$ is a unit of $\mathcal R$, so the coefficients of $P$ generate the unit ideal. McCoy's annihilator theorem \cite{McCoy1957}, applied after clearing Laurent powers to the trivial extension $\mathcal R\ltimes(\mathscr S/\mathscr N)$, shows that multiplication by $P$ is injective. This contradicts \eqref{eq:paired-layer-module} and $U\ne0$. Hence $\lambda=0$.
\end{proof}

\medskip

\begin{proof}[Proof of Theorem~\ref{thm-zariski}: hard case]
    Let $d=2q$ for some positive integer $q$. We prove it by contradition. Assume that $k\le \lfloor d/2 \rfloor=q$, then it forces equality in \eqref{eq:port-cycle-dimension} and \eqref{eq:port-cycle-matrix-bounds}. Thus
    \begin{equation*}
        \dim W=\operatorname{rank}M=q,\quad \|M\|^2=4q,\quad \operatorname{tr}M=\sum_{\nu=1}^{2q} \sigma_\nu(M)=2q,
    \end{equation*}
    where $\sigma_1(M),\cdots,\sigma_{2q}(M)$ are the singular values of $M$. The construction in Proposition \ref{prop:port-cycle} also gives
    \begin{equation*}
        M_{ii}=1 \quad  \text{and}\quad \sum_{j\neq i}|M_{ij}|\le 1
    \end{equation*}
    for each $i$. Then
    \begin{equation*}
        \sum_{j}|M_{ij}|^2=1+\sum_{j\neq i}|M_{ij}|^2 \le 1+\left( \sum_{j\neq i}|M_{ij}|\right)^2\le 2.
    \end{equation*}
    The sum of the left-hand sides over all rows is $\|M\|^2=4q=2d$. Hence the equality holds for every $i$ and 
    \begin{equation*}
        \sum_{j\neq i}|M_{ij}|=\sum_{j\neq i}|M_{ij}|^2=1.
    \end{equation*}
    These identities force exactly one off-diagonal entry in each row to be nonzero with absolute value one. Thus there exist a map
    \begin{equation*}
        \pi: \lbrace 1,\cdots, 2q\rbrace\to \lbrace 1,\cdots,2q\rbrace
    \end{equation*}
    and signs $\epsilon_i\in \lbrace \pm 1\rbrace$ such that
    \begin{equation*}
        \operatorname{row}_i(M)= e_i+\epsilon_i e_{\pi(i)} \quad \text{with}\quad \pi(i)\neq i. 
    \end{equation*}
    By the singular value decomposition, the matrix $M$ can be written as $M=U\Sigma Q^T$ where $U$ and $Q$ are orthogonal, $\Sigma=\operatorname{diag}\lbrace \sigma_1(M),\cdots,\sigma_{2q}(M) \rbrace$.  Then
    \begin{equation*}
        2q=\operatorname{tr}M=\sum_{\nu=1}^{2q}\sum_{\eta=1}^{2q} U_{\nu,\eta} \sigma_\eta(M) Q^T_{\eta,\nu}= \sum_{\sigma_\eta(M)>0} \sigma_\eta (M)\langle U_{\cdot,\eta},Q_{\cdot,\eta}\rangle\le \sum_{\sigma_\eta(M)>0}\sigma_\eta(M) =2q
    \end{equation*}
    where we used the fact that $U_{\cdot,\eta}$ and $Q_{\cdot,\eta}$ are unit vectors. The equality forces
    \begin{equation*}
        \langle U_{\cdot,\eta},Q_{\cdot,\eta}\rangle=1
    \end{equation*}
    for all $\eta$ satisfying $\sigma_\eta(M)>0$. Therefore 
    \begin{equation*}
        M_{ij}=\sum_{\sigma_\nu(M)>0}U_{i,\nu}\sigma_\nu (M) U_{j,\nu}= M_{j,i},
    \end{equation*}
    that is, $M$ is symmetric and positive semidefinite.

    By symmetry of $M$, we obtain
    \begin{equation*}
        \pi(\pi(i))=i \quad \text{and}\quad \epsilon_{\pi(i)}=\epsilon_i.
    \end{equation*}
    Relabel the coordinates so that its pairs become $\lbrace a_j,b_j\rbrace$, $1\le j\le q$. Set $\epsilon_j:=\epsilon_{a_j}=\epsilon_{b_j}$. Since all rows of $M$ generate $W$, we obtain
    \begin{equation}\label{eq:paired-space}
        W=\operatorname{span}\lbrace e_{a_j}+\epsilon_je_{b_j}:1\le j\le q\rbrace.
    \end{equation}
    Fix $x_0\in V$. Since $x_0+W\subset V$ and both irreducible varieties have dimension $q$, $V$ must be an affine plane $x_0+W$. For each $V\in\mathcal{C}_k$, we obtain the same result. Thus $\mathcal{C}_k$ is the family of $q$-dimensional affine plane.

    Let $W$ be as in \eqref{eq:paired-space}. Let $\mathcal{C}_W\subset \mathcal{C}_k$ be the nonempty subfamily of components parallel to $W$ and set
    \begin{equation*}
        Y=\bigcup_{L\in \mathcal{C}_W}L.
    \end{equation*}
    Define
    \begin{equation*}
        w(x)=\begin{cases}u(x),&x\in Y,\\0,&x\notin Y,\end{cases}
    \end{equation*}
    and $v:=u-w$.
    Then $P_\lambda w=-P_\lambda v$. Let $K_0:=\lbrace 0, \pm e_1,\cdots,\pm e_{2q}\rbrace$ and define the finite unions
    \begin{equation*}
        \mathscr{A}=\bigcup_{\substack{L\in\mathcal C_W\\h\in K_0}}(L+h),
 \quad
 \mathscr{B}=
 \bigcup_{\substack{Z\in \mathcal{C},\ Z\notin\mathcal C_W\\h\in K_0}}(Z+h).
    \end{equation*}
    Then we have
    \begin{equation*}
        \supp(P_\lambda w)\subseteq\mathscr A\quad \text{and}\quad\supp(P_\lambda v)\subseteq\mathscr B.
    \end{equation*}
    Since $P_\lambda w=-P_\lambda v$, their common support is contained in $\mathscr{A}\cap \mathscr{B}$, which is a finite union of intersections $(L+h)\cap (Z+h')$ with $L\in \mathcal{C}_W$ and $Z\notin \mathcal{C}_W$. If $\dim Z=q$, then $Z$ must be an affine plane with direction different from $W$, therefore the intersection has dimension less than $q$. If $\dim Z<q$, the same conclusion is immediate. Thus
    \begin{equation}\label{eq:small-error}
        \dim_{\C}\overline{\supp(P_\lambda w)}^{\,\mathrm{Zar}}<q.
    \end{equation}
    On the one hand, the support of $w$ lies in the finite union $Y$ of affine planes parallel to $W$. On the other hand, $S\cap V^\circ\subset \supp (w) $ is dense in $V$, hence
    \begin{equation*}
        \dim_{\C}\overline{\supp(w)}^{\,\mathrm{Zar}}=q.
    \end{equation*}
    Together with \eqref{eq:small-error}, Lemma~\ref{lem:paired-layer-rigidity} gives $\lambda=0$, which gives a contradiction. Thus $k>q$ and therefore we finish the proof.
\end{proof}

\end{document}